\documentclass[9pt,a4paper]{amsart}

\usepackage{xcolor}
\usepackage[unicode,bookmarks, pdftex, backref = page]{hyperref}
\definecolor{newblue}{RGB}{0,102,204}
\definecolor{newred}{RGB}{206,32,41}
\hypersetup{colorlinks=true,citecolor=newblue,linkcolor=newred,
  urlcolor=magenta,
pdfpagemode=UseNone, breaklinks=true}

\theoremstyle{plain}
\newtheorem{theorem}{Theorem}[section]
\newtheorem{lemma}[theorem]{Lemma}
\newtheorem{proposition}[theorem]{Proposition}
\newtheorem{corollary}[theorem]{Corollary}

\theoremstyle{definition}

\theoremstyle{remark}
\newtheorem{remark}[theorem]{Remark}

\newtheorem{notations}[theorem]{Notations}

\def\bin #1#2 {\left( \matrix { #1 \cr #2 \cr } \right) }

\begin{document}

\title[On the genus of a curve in a projective 3\,-fold]
{On the genus of a curve in a projective 3\,-fold}

\author{Vincenzo Di Gennaro }
\address{Universit\`a di Roma \lq\lq Tor Vergata\rq\rq, Dipartimento di Matematica,
Via della Ricerca Scientifica, 00133 Roma, Italy.}
\email{digennar@mat.uniroma2.it}

\author{Antonio Rapagnetta}
\address{Universit\`a di Roma \lq\lq Tor Vergata\rq\rq, Dipartimento di Matematica,
Via della Ricerca Scientifica, 00133 Roma, Italy.}
\email{rapagnet@mat.uniroma2.it}

\author{Pietro Sabatino}
\address{Institute for high performance computing and networking (ICAR-CNR),
via P. Bucci 8/9C 87036 Rende (CS), Italy.}
\email{pietro.sabatino@icar.cnr.it}

\abstract Let $X\subset \mathbb P^r$ be a projective factorial
variety of dimension $3$, degree $n$, with at worst isolated
singularities. Assume that the Picard group of $X$ is generated by
the hyperplane section class. Let $C\subset X$ be a projective
subscheme of dimension $1$, degree $d\gg n$, and arithmetic genus
$p_a(C)$. Improving a recent result by Liu, we exhibit a Castelnuovo's
bound for $p_a(C)$. In the case $X$ is Calabi-Yau, our bound gives a
step forward for a certain conjecture concerning the vanishing of
Gopakumar-Vafa invariants of $X$.



\endabstract
\maketitle

\section{Introduction}

A classical problem in the theory of
projective curves is to determine  all their possible genera in
terms of the degree and the dimension of the space where they are
embedded \cite{EH}, \cite{H}. Although a primary focus of interest
has been, and still is, on irreducible and reduced curves, there are
good reasons for studying more general curves, i.e. projective
schemes of dimension $1$. For instance, in the enumerative geometry
there is some interest, coming from theoretical physics, in certain
invariants of Calabi-Yau 3\,-folds $X$, as Gromov-Witten invariants and
Gopakumar-Vafa invariants, involving $1$-cycles of $X$  \cite{Liu}.

\smallskip
$\bullet$
Before stating our results,
we believe it is useful to briefly summarize some classical results
on the genus of a projective curve.

\smallskip
Let $C\subset \mathbb P^2$ be a projective curve of degree $d$. The
arithmetic genus $$p_a(C):=1-\chi(\mathcal O_C)$$ of $C$ is
determined by the degree:
\begin{equation}\label{genusplane}
p_a(C)=\pi_0(d,2):=\binom{d-1}{2}=\frac{d^2}{2}-\frac{3}{2}d+1.
\end{equation}

\smallskip
This is no longer true when $C$ is not contained in a plane.
However, there is an upper bound for $p_a(C)$. More precisely,
assume that $C\subset \mathbb P^r$, $r\geq 3$, is an {\it integral},
i.e. reduced and irreducible, projective curve of degree $d$. Assume
moreover that $C$ is non degenerate, i.e. not contained in a hyperplane.
Then the celebrated {\it Castelnuovo's bound} states that:
\begin{equation}\label{Sabattino}
p_a(C)\leq
\pi_0(d,r):=\frac{d^2}{2(r-1)}-\frac{d}{2(r-1)}(r+1)+\frac{1+\epsilon}{2(r-1)}(r-\epsilon),
\end{equation}
where $\epsilon$ is defined by dividing $d-1=m(r-1)+\epsilon$,
$0\leq \epsilon\leq r-2$ \cite[Theorem (3.7), p. 87]{EH}.
Notice that for every $r\geq 3$ and $d\geq
r$ one has
\begin{equation}\label{decr}
\pi_0(d,r)<\pi_0(d,r-1).
\end{equation}
Castelnuovo's bound is sharp, and there is a classification of the
extremal curves, the so called {\it Castelnuovo curves}. It turns
out that {\it the Castelnuovo curves must lie on quadric
hypersurfaces}, and that, at least when $d>2r$, must lie on  {\it
rational normal surfaces of degree $r-1$ in $\mathbb P^r$}
\cite[Theorem (3.11), p. 91]{EH}.

\smallskip
Now assume that $C$ is contained in $\mathbb P^3$, i.e. put $r=3$.
We just said that if $C$ is a Castelnuovo curve, it must lie on a
quadric. Therefore, the arithmetic genus of an integral curve of
degree $d$ {\it not} lying on a quadric should be strictly less than
$\pi_0(d,3)$. So, it is natural to ask what is the sharp bound for a
curve not contained in quadrics. More generally, given an integer $s
\geq 2$, one can ask {\it what is the sharp bound for the genus of a
curve in $\mathbb P^3$ not contained in surfaces of degree $<s$}. We
will call this question {\it Halphen's problem}. When
$s=2$, the estimate sought is precisely the Castelnuovo's bound. The
celebrated Halphen-Gruson-Peskine Theorem  gives an answer to this
problem \cite[Th\'eor\`eme 3.1]{GP}, \cite[10.8.Teorema]{Ellia}:
\begin{theorem}[Halphen-Gruson-Peskine]\label{HGP}
Fix an integer $s\geq 3$. Let $C\subset \mathbb P^3$ be a projective
integral curve of degree $d>s^2-s$, not contained in a surface of
degree $<s$. Then one has:
\begin{equation}\label{HGPbound}
p_a(C) \leq \frac{d^2}{2s}+\frac{d}{2}(s-4)+1-\epsilon_1,
\end{equation}
where $$\epsilon_1=\frac{(s-1-\epsilon)(1+\epsilon)(s-1)}{2s},$$ and
$\epsilon$ is defined by dividing $d-1=ms+\epsilon$, $0\leq
\epsilon\leq s-1$.
\end{theorem}
The inequality (\ref{HGPbound}) is sharp, and there is a description
of the extremal curves, which are necessarily contained in a surface
of degree $s$ \cite{GP}, \cite{Ellia}. Observe that
\begin{equation}\label{cfrLiu}
\frac{1}{2}(s-4)=\frac{1}{2s}(2\pi_0(s,2)-2-s).
\end{equation}
Theorem \ref{HGP} can be generalized. In fact, one may prove that if
$C\subset \mathbb P^r$, $r\geq 3$, is a projective integral non degenerate curve of
degree $d\gg s$, not contained in a surface of degree $<s$, then one
has:
\begin{equation}\label{MainCCD}
p_a(C) \leq \frac{d^2}{2s}+\frac{d}{2s}(2\pi_0-2-s)+O(1),
\end{equation}
where $O(1)$ is a constant, which only depends on $r$ and $s$, that
can be described explicitly, as is the hypothesis $d\gg s$.
Moreover, $\pi_0=\pi_0(s,r-1)$ denotes the Castelnuovo's bound for
curves of degree $s$ in $\mathbb P^{r-1}$ (see
\cite{EH}, \cite{Harris2}, \cite[Main Theorem]{CCD}). The bound is sharp, and
the extremal curves must lie on surfaces $S$ of degree $s$ with
sectional genus $\pi_0(s,r-1)$ (the sectional genus is the arithmetic
genus of a hyperplane section of $S$). Such surfaces are called {\it
Castelnuovo surfaces}.
When $r=3$, (\ref{MainCCD}) reduces to Halphen bound (\ref{HGPbound}).

\smallskip
Halphen's problem admits a weaker formulation. One can ask what is
the sharp bound for the arithmetic genus of an integral curve of
degree $d$, contained in a {\it fixed} integral surface of $\mathbb
P^3$ of degree $s$. This is {\it Noether's problem}. This problem
can be seen as an extension to any surface $S$ of what happens when
$S$ is a plane. If $d>s^2-s$, in view of Bezout's theorem, a curve
contained in a surface of degree $s$ cannot be contained in a
surface of degree $<s$, and so we fall back into Halphen's problem.
And in fact, the answer given to this problem is exactly the same as
Halphen's estimate \cite{Harris1980}. Actually, one may examine {Noether's
problem} also in the range $d\leq s^2-s$ \cite{Harris1980}.

\smallskip
Noether's problem can also be formulated for a fixed
surface in $\mathbb P^r$, $r\geq 3$. In this case, one knows the
following. If $C$ is an integral curve of degree $d\gg s$, contained
in a given integral non degenerate surface $S\subset \mathbb P^r$, $r\geq 3$, of
degree $s$ and sectional genus $\pi$, one has
\begin{equation}\label{LemmaVDG}
p_a(C)\leq \frac{d^2}{2s}+\frac{d}{2s}(2\pi-s-2)+O(1),
\end{equation}
where $O(1)$ is a constant which only depends on $s$ and $r$, but
not sharp \cite[Lemma]{D3}. However, we notice that if $C=S\cap F_{m+1}$ is a
complete intersection of $S$ with a hypersurface $F_{m+1}$ of degree
$m+1=\frac{d}{s}$ then (compare with (\ref{Sab2}) below):
\begin{equation}\label{boundintcompl}
p_a(C)=\frac{d^2}{2s}+\frac{d}{2s}(2\pi-s-2)+1.
\end{equation}
This means that the previous estimate (\ref{LemmaVDG}), excluding
the constant term $O(1)$, is sharp. Moreover, we also notice that,
based on an observation by Harris \cite[p. 196]{Harris1980}, the previous
bound (\ref{LemmaVDG}) is in a certain sense predicted by the Hodge
Index Theorem. In fact, if $S\subseteq \mathbb P^r$ is a smooth
surface, then from the Hodge Index Theorem we know that
$$\left(C-\frac{d}{s}H_S\right)^2\leq 0,$$
$H_S$\,=\,hyperplane section of $S$. If $S$ is also subcanonical,
that is $K_S=qH_S$, $q\in\mathbb Q$, applying the adjunction formula
we deduce
\begin{equation}\label{HIT}
p_a(C)\leq \frac{d^2}{2s}+\frac{d}{2s}(2\pi-2-s)+1
\end{equation}
for every closed subscheme $C$ of $S$ of dimension $1$ and degree
$d$. Here, $\pi$ denotes the sectional genus of $S$. Notice that, in view of Castelnuovo's bound,
if $S$ is non degenerate in  $\mathbb P^r$ with $r\geq 4$, then (compare with (\ref{decr}) and (\ref{cfrLiu})):
\begin{equation}\label{lessLiu}
\frac{1}{2s}(2\pi-s-2)\leq \frac{1}{2s}(2\pi_0(s,r-1)-2-s)< \frac{1}{2s}(2\pi_0(s,2)-2-s)=\frac{1}{2}(s-4).
\end{equation}

\smallskip
Halphen-Gruson-Peskine Theorem, and all  the subsequent quoted
generalized results, apply to integral curves. Moreover, almost their proofs
are based on a property satisfied by the general hyperplane section
$\Gamma$ of an integral curve, {\it the uniform position principle}
(see for example \cite{Harris1980}, \cite{EH}, \cite{CCD}, \cite{Ellia}). Simply
put, {\it it is not possible to distinguish two subsets of points of
$\Gamma$ having the same order}.

\smallskip
$\bullet$
Now let's turn to more recent results, and those appearing in our
article.

\smallskip
Recently, Macr\`i and Schmidt observed that
Theorem \ref{HGP}
can be proved with a new method,
called {\it wall-crossing}, which uses the {\it Bridgeland stability
conditions} \cite{Macri'S}. In this new proof a generalized  Bogomolov-Gieseker inequality,
concerning  Chern classes of certain sheaves on
$\mathbb P^3$ and necessary to construct Bridgeland stability conditions on $\mathbb P^3$,  is fundamental \cite{Macri'}.

Next, based on the work of Macr\`i and Schmidt, Liu studies the genus
of a curve, even if it is not integral, on certain projective
3\,-folds, obtaining a sort of relative Castelnuovo's bound:
\begin{theorem}\cite[Corollary 4.15]{Liu}\label{Liu}
Let $X\subseteq\mathbb P^r$ be a projective factorial  variety of
dimension $3$, degree $n$, with at
worst isolated singularities. Assume that the Picard group of $X$ is
generated by the hyperplane section class. Let $C\subset X$ be a projective
subscheme of dimension $1$, and degree $d$. If $d\gg n$, then
\begin{equation}\label{Liubound}
p_a(C)\leq \frac{d^2}{2n}+\frac{d}{2}(n-4)+1-\epsilon_1.
\end{equation}
\end{theorem}
The number $\epsilon_1$ appearing in (\ref{Liubound})   is defined
as in (\ref{HGPbound}) with $n$ instead of $s$. And in fact, setting
$n=s$, the bound (\ref{Liubound}) is equal to  (\ref{HGPbound}).
Actually, previous theorem is stated under slightly more general
hypotheses, and one may make explicit the assumption $d\gg n$, see
\cite[Theorem 4.9, Theorem 4.10, and Corollary 4.11]{Liu}. Observe that when
$X=\mathbb P^3$, the bound (\ref{Liubound}) reduces to
(\ref{genusplane}).

\smallskip
In short, the proof consists of a iterative method, whose purpose is to
reduce, via wall-crossing,
to a Noether's problem for curves on a {\it fixed} surface of $X$ of
small degree compared to $d$. Here, however, curves and surfaces
may not be integral. Therefore, to examine this case,
one still need to use the wall-crossing method.
Next, one reduces it to an
integral surface of $X$, and one may conclude  with the generalized
Bogomolov-Gieseker inequality.
Note that, if the curve $C$ is
integral, Theorem \ref{Liu} can be easily proved with
classical methods by combining \cite{CCD} with \cite{D3} (see
Remark \ref{finalr}, $(iii)$, below).

\smallskip
To apply the wall-crossing method and the generalized
Bogomolov-Gieseker inequality, one must reduce the problem to a
curve problem in $\mathbb P^3$. This can be done by projecting
$C$ onto a curve $C'$ of $\mathbb P^3$ of the same degree, such
that $p_a(C)\leq p_a(C')$. Therefore, the problem reduces to
$C'$. This explains why the bound
(\ref{Liubound}) is equal to Halphen's bound (\ref{HGPbound}) in $\mathbb P^3$.

\smallskip
But Hodge Index Theorem suggests that, in general,  the linear term
in $d$ in the bound (\ref{Liubound}) should not be fine (compare
with (\ref{HIT}) and (\ref{lessLiu})).

\smallskip
$\bullet$
Taking into account inequality (\ref{HIT}),
our paper aims to refine
the bound (\ref{Liubound}) of Theorem \ref{Liu} in the linear term in $d$.
As already noted in (\ref{boundintcompl}), the linear term in $d$ suggested by the
Hodge Index Theorem is sharp.

\smallskip
Our main result is the following:
\begin{theorem}\label{n3fold}
Let $X\subseteq\mathbb P^r$ be a projective factorial  variety of
dimension $3$, degree $n$, linear arithmetic genus $\pi_X$, with at
worst isolated singularities. Assume that the Picard group of $X$ is
generated by the hyperplane section class.  Set:
$$\nu=\frac{4}{3}n(n+1), \quad M(n)=\frac{n^2}{2}\left[\frac{2}{3}(n+1)+4+2\nu^6\right]^2.$$
Let $C\subset X$ be a closed subscheme of $X$ of dimension $1$, of degree $d>M(n)$, and arithmetic genus
$p_a(C)$. Then one has:
\begin{equation}\label{nbound3fold}
p_a(C)\leq \frac{d^2}{2n}+\frac{d}{2n}(2\pi_X-2-n)+4\nu^6.
\end{equation}
\end{theorem}
The linear arithmetic genus $\pi_X$ of $X$ is the sectional genus
of its general hyperplane section. When $d\gg n$, the bound (\ref{nbound3fold}) is finer than the bound
(\ref{Liubound}) (compare with (\ref{lessLiu})).

\smallskip
The proof of Theorem \ref{n3fold} consists in showing that the
iterative method introduced in \cite[Proof of Theorem 4.9]{Liu} also
fits our estimate (\ref{nbound3fold}). However, once the problem is
reduced to the case where the curve is contained in a surface $S$ of
$X$ of degree $s$ small compared to $d$, our proof proceeds a little
differently. In fact, to solve this Noether's problem (up to the
constant term), the wall-crossing method is not necessary. With
classical arguments, we reduce to the case where $S$ is integral,
and we obtain the corresponding estimate by adapting the argument appearing
in \cite[Proof of Lemma]{D3}. To this purpose, see Proposition \ref{nSintegra}
and Proposition \ref{nSnonintegra} below, and their proofs. In the
subsequent Proposition \ref{refined}, Corollary \ref{intcompl},
Corollary \ref{CY} and Corollary \ref{P3}, we refine the constant
term of the bound appearing in Proposition \ref{nSnonintegra} under
the assumption $d\gg \deg S$, give some information on extremal
curves contained in $S$, and in certain cases we prove sharp bounds.
In particular, we prove that extremal curves are contained in
components of $S$ of minimal degree.

\smallskip
Taking into account these Noether's type results, when $d\gg M(n)$,
we are able to refine the constant term in (\ref{nbound3fold}), and
to prove that the extremal curves in $X$ are contained in a
hyperplane section. For more precise statements, see Corollary
\ref{cor1} and Corollary \ref{cor2} below.

\smallskip
As a consequence of Corollary \ref{cor1}, we obtain the following
result, concerning Calabi-Yau varieties of dimension $3$:
\smallskip
\begin{corollary}\label{ncor3}
Let $X\subseteq\mathbb P^r$ be a smooth Calabi-Yau variety of
dimension $3$ and degree $n$. Assume that the Picard group of $X$ is
generated by the hyperplane section class. Let $C$ be  a closed subscheme of $X$ of dimension $1$
and of degree $d$ with $d>M^*(n)=6n^2M(n)^2$. Then we have:
$$
p_a(C) \leq \frac{d^2}{2n}+\frac{d}{2}+n^4.
$$
\end{corollary}
\smallskip
Corollary \ref{ncor3} appears to us to be a significant step forward
for the Castelnuovo bound conjecture
stated in \cite[Conjecture 1.1]{Liu}.

\smallskip
As for the  numerical assumption  $d\gg n$ in Theorem \ref{n3fold},
and the constant term appearing in the bound (\ref{nbound3fold}),
they are certainly not the best possible. One might hope to do
better with a closer examination of the numerical functions arising
in the proofs. We decided not to push here this investigation
further. The same remark applies to the assumption that the Picard
group is generated by a {\it very} ample divisor, and to other
numerical assumptions appearing in our subsequent results.

\smallskip
In the sequel we use standard notations in Algebraic Geometry, see
e.g. \cite{Hartshorne2}. Our base field is $\mathbb C$. By {\it
curve} we mean  a projective scheme of dimension $1$ (in \cite{Liu}
{\it curve} means a Cohen-Macaulay projective scheme of dimension
$1$).

\section{Bound for the genus of a curve on a surface.}

In this section we establish a Noether's type bound for the arithmetic
genus $p_a(C)$ of a projective scheme $C$ of dimension $1$ contained
in a projective scheme $S$ of dimension $2$. When $S$ is integral
(i.e. reduced and irreducible) we do not need any other hypothesis
on $S$. Otherwise, we have to assume that $S$ is contained in an
integral projective threefold $X\subseteq \mathbb P^r$, and that here $S$
is a Cartier divisor whose irreducible components are linearly
equivalent to a multiple of the hyperplane section of $X$ (see
Notations \ref{secondnot} below). The bound we prove, apart from the
constant term, is suggested by the Hodge Index Theorem (compare with
(\ref{HIT}) and see also Remark \ref{Halphen} below), it is sharp (compare with (\ref{boundintcompl})),
and it is more precise than the one that appears in \cite[Corollary 4.4]{Liu}  (see (\ref{lessLiu})).
First, in Proposition \ref{nSintegra} below, we examine the case $S$
is integral (i.e. reduced and irreducible). The proof relies on the
same argument developed in \cite[Lemma]{D3}.
This argument is classic in nature and does not require the wall-crossing method
(compare with \cite[Proof of Proposition 4.3]{Liu}).
Next, we argue by induction on the number of components of $S$.
Before stating our
result, we need some preliminaries.

\smallskip
\begin{notations}\label{primenotazioni}
$(i)$ By {\it curve} we mean  a projective scheme of dimension $1$.

\smallskip
$(ii)$ If $C$ is a curve, we define the {\it arithmetic genus
$p_a(C)$ of $C$} as $p_a(C)=1-\chi(\mathcal O_C)$. Observe that when
$C$ is integral then $p_a(C)\geq 0$.

\smallskip
$(iii)$ Let $Z\subseteq \mathbb P^r$ be a projective subscheme of
$\mathbb P^r$. We denote by $h_Z(i)$ the Hilbert function of $Z$ in
$\mathbb P^r$ at level $i\in \mathbb Z$. In other words
$$
h_Z(i)=\dim\left[{\text{Image}} \left(H^0(\mathbb P^r,\mathcal
O_{\mathbb P^r}(i))\to H^0(Z,\mathcal O_{Z}(i))\right)\right].
$$
Observe that $h_{Z}(i)=0$ when $i<0$, and that if $Z\subseteq
\mathbb P^{r-1}$ then the Hilbert function of $Z$ in $\mathbb P^r$
is equal to the Hilbert function of $Z$ in $\mathbb P^{r-1}$.

\smallskip
$(iv)$ Assume that  $\Gamma$ is a zero-dimensional subscheme of
$\mathbb P^r$ of length $d>0$. We have $h_{\Gamma}(0)=1$. Moreover,
$h_{\Gamma}(i)\leq d$ for every $i\in\mathbb Z$, and
$h_{\Gamma}(i)=d$ for $i\gg 0$. Set
$$
i_0=i_0(\Gamma)=\min\left\{i\in\mathbb Z\,:\,
h_{\Gamma}(i)=d\right\}.
$$
Observe that $i_0\geq 0$.
\end{notations}

\smallskip
\begin{lemma}\label{hilbseziperp}
Let $\Gamma$ be a zero-dimensional subscheme of $\mathbb P^r$ of
length $d>0$. Then $h_{\Gamma}(i)< h_{\Gamma}(i+1)$ for every $0\leq
i<i_0$. And $h_{\Gamma}(i)=d$ for every $i\geq i_0$. In particular,
$h_{\Gamma}(d-1)=d$.
\end{lemma}

\begin{proof} When $r=2$, these properties follow from \cite[Proof of Proposition 1.1 and Corollary
1.2]{H}. The same argument of the proof applies also to the case
$r\geq 3$.
\end{proof}

Next Lemma \ref{riduzioneseziperp} appears in \cite[Corollary 3.2, p. 84]{EH}. We
prove it for Reader's convenience.

\smallskip
\begin{lemma}\label{riduzioneseziperp}
Let $C\subseteq \mathbb P^r$ be a curve of degree $d>0$, and let
$\Gamma$ be a general hyperplane section of $C$. Then we have:
$$
p_a(C)\leq \sum_{i=1}^{+\infty}d-h_{\Gamma}(i).
$$
\end{lemma}

\begin{proof}
Fix an integer $j \gg 0$.  We have:
$$
1-p_a(C)+jd=h^0(C,\mathcal O_C(j))-h^1(C,\mathcal
O_C(j))=h^0(C,\mathcal O_C(j))=h_{C}(j).
$$
On the other hand, by \cite[Lemma 3.1, p. 83]{EH} we know that for
every $i\in \mathbb Z$
$$
h_{C}(i)-h_{C}(i-1)\geq h_{\Gamma}(i).
$$
Therefore, we have
$$
h_C(j)=\sum_{i=0}^{j}h_C(i)-h_C(i-1)\geq
\sum_{i=0}^{j}h_{\Gamma}(i).
$$
Taking into account that $h_{\Gamma}(0)=1$ and that
$h_{\Gamma}(i)=d$ for $i\geq j$, it follows that
$$
p_a(C)=1+jd-h_C(j)\leq
1+jd-\sum_{i=0}^{j}h_{\Gamma}(i)=\sum_{i=1}^{j}d-h_{\Gamma}(i)=\sum_{i=1}^{+\infty}d-h_{\Gamma}(i).
$$
\end{proof}

\smallskip
\begin{lemma}\label{planegenus}
Let $C\subseteq \mathbb P^r$ be a curve of degree $d>0$. Then we
have:
$$
p_a(C)\leq \binom{d-1}{2}.
$$
\end{lemma}

\begin{proof} Combining Lemma \ref{riduzioneseziperp} with Lemma
\ref{hilbseziperp} we get:
$$
p_a(C)\leq \sum_{i=1}^{+\infty}d-h_{\Gamma}(i)\leq
\sum_{i=1}^{d-1}d-i-1=\binom{d-1}{2}.
$$
\end{proof}

\medskip
\begin{lemma}\label{trecarte} Let $d\geq 1$, $s\geq 2$, $\pi\geq 0$ and $\varphi\geq 1$ be integers. Then we have:
$$
d\leq \sqrt{2\varphi}\quad \implies \quad \binom{d-1}{2}\leq
\frac{d^2}{2s}+\frac{d}{2s}(2\pi-2-s)+\varphi.
$$
\end{lemma}

\begin{proof} Assume  $d\leq \sqrt{2\varphi}$. Since $s\geq 2$ and $\pi\geq 0$ we
have
$$
\frac{1}{2s}(2\pi-2-s)\geq -1.
$$
Therefore, it suffices to prove that
$$
\binom{d-1}{2}\leq \frac{d^2}{2s}-d+\varphi.
$$
This is equivalent to prove that
\begin{equation}\label{final}
d^2\frac{s-1}{2s}\leq \frac{1}{2}d-1+\varphi.
\end{equation}
Therefore, when $d\geq 2$, it suffices to prove that
$$
d^2\leq \varphi \frac{2s}{s-1}.
$$
And this holds true because
$$
2\varphi \leq\varphi \frac{2s}{s-1}.
$$
If $d=1$ then (\ref{final}) follows by direct substitution, taking
into account that $\varphi\geq 1$.
\end{proof}

\smallskip
We are in position to prove our bound for the genus of a curve
contained in a reduced, irreducible and projective surface
$S\subseteq \mathbb P^r$. The {\it sectional genus} of such a
surface is the arithmetic genus of the integral curve $S\cap H$,
where $H\subseteq \mathbb P^r$ is a general hyperplane.

\medskip
\begin{proposition}\label{nSintegra} Let $S\subseteq \mathbb P^r$ be
a reduced, irreducible and projective surface, of degree $s$ and
sectional genus $\pi$. Let $C\subset S$ be a curve of degree $d>0$.
Then
\begin{equation}\label{bound0}
p_a(C)\leq \frac{d^2}{2s}+\frac{d}{2s}(2\pi-2-s)+s^4.
\end{equation}
\end{proposition}

\begin{proof} If $s=1$, then (\ref{bound0}) follows from Lemma \ref{planegenus}.
So, we may assume $s\geq 2$. We may assume that $S$ is non degenerate
in $\mathbb P^r$. Therefore, we may also assume that $r\geq 3$.
Observe that if $s\geq 3$ and $d\leq s^2+s$ then $d\leq
\sqrt{2s^4}$. Combining Lemma \ref{planegenus} with Lemma
\ref{trecarte} we deduce
$$
p_a(C)\leq \binom{d-1}{2}\leq
\frac{d^2}{2s}+\frac{d}{2s}(2\pi-2-s)+s^4.
$$
By direct computation we see that previous inequality holds true
also when $s=2$ and $d\leq s^2+s=6$. Therefore, we may assume also
that $d>s^2+s$. This will allow us to apply \cite[Lemma]{D3} in a
little while.

\smallskip
Let $H\subset \mathbb P^r$ be a general hyperplane. Set:
$$
\Sigma=S\cap H\quad{\text{and}}\quad \Gamma=C\cap H.
$$
Observe that $\Sigma$ is a non degenerate integral curve of degree
$s$ in $H\cong \mathbb P^{r-1}$, and that $\Gamma$ is a
zero-dimensional scheme of length $d$. Since $\Gamma\subset \Sigma$
we have $h_{\Gamma}(i)\leq h_{\Sigma}(i)$ for every $i$. In view of
Lemma \ref{riduzioneseziperp} it suffices to prove that
$$
\sum_{i=1}^{+\infty}d-h_{\Gamma}(i)\leq
\frac{d^2}{2s}+\frac{d}{2s}(2\pi-2-s)+s^4.
$$
To this purpose, divide
\begin{equation}\label{divide}
d-1=ms+\epsilon,\quad 0\leq \epsilon\leq s-1.
\end{equation}
Since $\Sigma$ is an integral curve, by degree considerations, we
have
\begin{equation}\label{Bezout}
h_{\Gamma}(i)=h_{\Sigma}(i)\quad {\text{for every $i\leq m$.}}
\end{equation}
Therefore, we may split
\begin{equation}\label{2.2}
\sum_{i=1}^{+\infty}d-h_{\Gamma}(i)=\sum_{i=1}^{m}d-h_{\Sigma}(i)+\sum_{i=m+1}^{+\infty}d-h_{\Gamma}(i).
\end{equation}
The first summand depends only on $d$ and $\Sigma$. By \cite[Proof
of Lemma, p. 795, line 8 from above]{D3} we know that
$$
\sum_{i=1}^{m}d-h_{\Sigma}(i)=\binom{m}{2}s+m(\epsilon +
\pi)-\sum_{i=1}^{+\infty}(i-1)(s-h_{\Sigma'}(i))-\delta_i),
$$
where $\Sigma'$ is a general hyperplane section of $\Sigma$, and
$\delta_i$ is the dimension of the kernel of the natural map
$H^1(H,\mathcal I_{\Sigma,H}(i-1))\to H^1(H,\mathcal
I_{\Sigma,H}(i))$. Replacing $m$ with $\frac{d-1-\epsilon}{s}$, and
taking into account (\ref{2.2}), we get
$$
\sum_{i=1}^{+\infty}d-h_{\Gamma}(i)=\frac{d^2}{2s}+\frac{d}{2s}(2\pi-2-s)+\rho,
$$
where
$$
\rho=\frac{1+\epsilon}{2s}(s+1-\epsilon-2\pi)-\sum_{i=1}^{+\infty}(i-1)(s-h_{\Sigma'}(i)))
+\sum_{i=1}^{+\infty}(i-1)\delta_i+\sum_{i=m+1}^{+\infty}d-h_{\Gamma}(i).
$$
So,
\begin{equation}\label{rho}
{\text{it suffices to prove that $\rho\leq s^4$.}}
\end{equation}
The first three summands depend only on $\Sigma$, and by
\cite[Lemma, (2.5), (2.7), (2.8)]{D3} we can estimate their sum:
\begin{equation}\label{3terms}
\left[\frac{1+\epsilon}{2s}(s+1-\epsilon-2\pi)-\sum_{i=1}^{+\infty}(i-1)(s-h_{\Sigma'}(i)))
+\sum_{i=1}^{+\infty}(i-1)\delta_i\right]\leq s^3.
\end{equation}
It remains to estimate the last term, the tail:
$$
\sum_{i=m+1}^{+\infty}d-h_{\Gamma}(i).
$$
Here we observe that, since $C$ may be reducible or not reduced, one
cannot estimate the Hilbert function $h_{\Gamma}$ via uniform
position as in the classical Castelnuovo Theory \cite{EH},
\cite{CCD}. We may overlap this difficulty using Lemma
\ref{hilbseziperp}. Observe also that $\Sigma$ is integral. Hence,
classical Castelnuovo Theory applies to $\Sigma$.

\smallskip
Since $d>s^2+s$, by Castelnuovo Theory we have \cite[Theorem, $(i)$, p. 492]{GLP},
\cite[Theorem (3.7), p. 87]{EH}:
$$
h^1(H,\mathcal I_{\Sigma, H}(m))=h^1(\Sigma,\mathcal
O_{\Sigma}(m))=0.
$$
Hence, we have:
$$
h_{\Sigma}(m)=h^0(\Sigma,\mathcal O_{\Sigma}(m))=h^0(\Sigma,\mathcal
O_{\Sigma}(m))-h^1(\Sigma,\mathcal O_{\Sigma}(m))=1-\pi+ms.
$$
By (\ref{Bezout}) we deduce:
$$
h_{\Gamma}(m)=h_{\Sigma}(m)=1-\pi+ms.
$$
Therefore, by Lemma \ref{hilbseziperp} we obtain
$$
h_{\Gamma}(m+j)=d\quad {\text{for every $j\geq \pi+s$.}}
$$
Hence,
$$
\sum_{i=m+1}^{+\infty}d-h_{\Gamma}(i)=\sum_{i=m+1}^{m+\pi+s}d-h_{\Gamma}(i).
$$
On the other hand, again by Lemma \ref{hilbseziperp}, for every
$j\geq 0$ we have
$$
d-h_{\Gamma}(m+j)\leq d-h_{\Gamma}(m)=\pi+\epsilon.
$$
Hence, taking into account Lemma \ref{planegenus} and
(\ref{divide}), we get:
$$
\sum_{i=m+1}^{+\infty}d-h_{\Gamma}(i)=\sum_{i=m+1}^{m+\pi+s}d-h_{\Gamma}(i)\leq
(\pi+s)(\pi+\epsilon)\leq (\pi+s)^2\leq
\left(\binom{s-1}{2}+s\right)^2.
$$
Combining with (\ref{rho}) and (\ref{3terms}), we are done. In fact,
since $s\geq 2$, we have
$$
\rho \leq  s^3+\left(\binom{s-1}{2}+s\right)^2\leq s^4.
$$
\end{proof}

\begin{remark}\label{Halphen}
Keep the notations of Proposition \ref{nSintegra}.

\smallskip
$(i)$ If $C=S\cap F$ is the complete intersection of $S$ with a
hypersurface $F$ of degree $m+1$, then $C$ is a curve of degree
$d=(m+1)s$, and
$$
p_a(C)=\frac{d^2}{2s}+\frac{d}{2s}(2\pi-2-s)+1
$$
(compare with (\ref{Sab2}) below).
This proves that, apart the constant term, the bound (\ref{bound0})
is sharp.

\smallskip
$(ii)$ If $S$ is smooth and subcanonical, i.e. the canonical bundle
of $S$ is numerically equal to a rational multiple of the hyperplane
section, using Hodge Index Theorem we obtain  \cite[p. 196]{Harris1980}
(compare with (\ref{HIT})):
$$
p_a(C)\leq \frac{d^2}{2s}+\frac{d}{2s}(2\pi-2-s)+1.
$$
It is an open question whether this bound holds true also when $S$
is an integral complete intersection surface, possibly singular,  $C$ is integral, and $d\gg s$ (compare
with \cite{D3}).

\smallskip
$(iii)$ Assume $S\subset \mathbb P^3$. In this case $\Sigma$ is a
plane curve of degree $s$. Therefore, we have:
$$
\pi=\binom{s-1}{2},
$$
and so (compare with \cite[Corollary 4.15]{Liu}):
$$
\frac{d}{2s}(2\pi-2-s)=\frac{d}{2}(s-4).
$$

\smallskip
$(iv)$ Assume $S\subset \mathbb P^3$. With the same notations as in
(\ref{divide}), set:
$$
\epsilon_1:=\frac{(s-1-\epsilon)(1+\epsilon)(s-1)}{2s}.
$$
Then, using Bridgeland Theory, one may prove \cite[Lemma 4.1]{Liu}:
$$
p_a(C)\leq \frac{d^2}{2s}+\frac{d}{2s}(2\pi-2-s)+1-\epsilon_1.
$$
In the case $C$ is integral and $d>s^2-s$, this is the celebrated
Halphen-Noether bound (see \cite{Harris1980}, \cite{GP}, and compare with
Theorem \ref{HGP}). On certain
surfaces, when $d>s^2-s$, the bound is sharp. In the case $C$ is
integral and $d\leq s^2-s$, the bound may be not sharp. It should be
corrected \cite{Harris1980}.


\smallskip
$(v)$ In view of Castelnuovo's bound (\ref{Sabattino}), if $\Sigma\subset
\mathbb P^{r-1}$ is a non degenerate integral curve of arithmetic
genus $\pi$, degree $s\geq r-1$, with $r-1\geq 3$, then
$$
\pi<\binom{s-1}{2}.
$$
Therefore, if $S\subset \mathbb P^r$ is non degenerate and $r\geq
4$, then (compare with (\ref{lessLiu})):
$$
\frac{d}{2s}(2\pi-2-s)<\frac{d}{2s}\left(2\binom{s-1}{2}-2-s\right)=\frac{d}{2}(s-4).
$$
This explains why, for $d\gg 0$, the bound (\ref{bound0}) is more precise
than the one that appears in \cite[Corollary 4.4]{Liu}.
\end{remark}

\bigskip
Next, we turn our attention to the case the surface $S$ is not
integral. In this case, we have to assume that $S$ is contained in an
integral projective threefold $X\subseteq \mathbb P^r$, and that $S$
is a Cartier divisor whose irreducible components are linearly
equivalent to a multiple of the hyperplane section of $X$.
Before stating our results, we need some preliminaries.

\begin{notations}\label{secondnot}
$(i)$ Fix a projective integral variety $X\subset \mathbb P^r$ of
dimension $\dim X=3$, and degree $\deg X=n$. Let $D_1,D_2,\dots,D_l$
effective Cartier divisors of $X$. We identify  $D_1,D_2,\dots,D_l$
with the associated closed subschemes of $X$. Therefore, we may
think of $D_1,D_2,\dots,D_l$ as  projective surfaces contained in
$X$. We assume every surface $D_i$ is integral, and that $D_i\neq
D_j$ when $i\neq j$. Set $s_i=\deg D_i$. We assume
$$
s_1\leq s_2\leq \dots\leq s_l.
$$
Set
$$
s=s_1 =\min\{s_1,s_2,\dots,s_l\}.
$$
We assume each $D_i$ linearly equivalent to a multiple of the
general hyperplane section $H$ of $X$:
\begin{equation}\label{le}
D_i\sim a_iH.
\end{equation}
Observe that $s_i=a_in$. We denote by $\pi_i$ the sectional genus of
the surface $D_i$, and by $\pi_X$ the arithmetic genus of the general
linear section of $X$ of dimension $1$, i.e. the sectional genus of $H$.
If $D_i'$ is the general hyperplane section of $D_i$, we have the natural exact sequence
$0\to \mathcal O_H(-a_i)\to \mathcal O_H\to \mathcal O_{D_i'}\to 0$. It follows that
\begin{equation}\label{Sab2}
\pi_i=\frac{s_i^2}{2n}+\frac{s_i}{2n}(2\pi_X-2-n)+1.
\end{equation}
In particular, we have
$$
s_i=s_j\quad\implies\quad \pi_i=\pi_j.
$$
Moreover, since $0\leq \pi_i\leq \binom {s_i-1}{2}$, for every $i$
we have
\begin{equation}\label{stimapi}
-2\leq \frac{2\pi_i-2-s_i}{2s_i}\leq s_i.
\end{equation}
Let $f_i$ be a local generator of the ideal sheaf $\mathcal I_{D_i}$
of $D_i$ in $X$. Fix positive integers $r_1,r_2,\dots,r_l$. Let $S$
be the effective Cartier divisor of $X$ with local generator given
by
$$
f_1^{r_1}\cdot f_2^{r_2}\cdot\dots\,\cdot f_l^{r_l}.
$$
We will use the notation
$$
S=r_1D_1+ r_2D_2+\dots+ r_lD_l.
$$
Set
$$
\deg S=r_1s_1+ r_2s_2+\dots+ r_ls_l.
$$
We identify $S$ with its associated subscheme of $X$. Therefore, we
may think of $S$ as a projective surface contained in $X$.

\smallskip
$(ii)$ Fix a surface $S$ as above, and an integer $d>0$. Let
$\mathcal C(S,d)$ ($\mathcal C^*(S,d)$ resp.) be the set of all
curves $C$ of degree $d$ contained in $S$ (in some integral
component $D_i$ of $S$ of minimal degree $s$ resp.). Set
$$
G(S,d)=\max\left\{p_a(C)\,:\, C\in \mathcal C(S,d)\right\}\quad
{\text{and}}\quad G^*(S,d)=\max\left\{p_a(C)\,:\, C\in \mathcal
C^*(S,d)\right\}.
$$

\smallskip
$(iii)$ If $Y$ is a closed subscheme of $X$, we denote by $\mathcal
I_Y$ its ideal sheaf in $X$.
\end{notations}

In the following we keep the notations and the assumptions just
stated.

\smallskip
In the first statement Proposition \ref{nSnonintegra} we prove a
bound for the arithmetic genus of  a curve of $S$ of degree $d>0$.
In the second statement Proposition \ref{refined} we refine previous
bound under the assumption $d\gg \deg S$. These results should be
compared with \cite[Proposition 4.3 and Corollary 4.4]{Liu}. Our
results, in general, are more precise than
\cite[loc. cit.]{Liu} (compare with (\ref{lessLiu})).  In fact, by Remark \ref{Halphen}, $(i)$,
apart the constant term, the bounds appearing in  Proposition
\ref{nSnonintegra} and Proposition \ref{refined} are sharp, at least
when $d$ is a multiple of the minimal degree $s$. Moreover, unlike
\cite{Liu}, our argument in the proofs do not need Bridgeland
Theory. The following  Corollary \ref{CY} and Corollary \ref{P3} are
immediate consequences of Proposition \ref{refined}, taking into
account Remark \ref{Halphen}, $(i)$, $(ii)$, $(iv)$. Also Corollary \ref{intcompl} is a
consequence  of Proposition \ref{refined}, apart from one detail, so
we give the proof. As for the the constant term $4(\deg S)^6$
appearing in the bound (\ref{bound1}) of Proposition
\ref{nSnonintegra},  the numerical assumption $d>30(\deg S)^7$
appearing in the claim of Proposition \ref{refined}, Corollary
\ref{intcompl}, Corollary \ref{CY} and Corollary \ref{P3}, and the
term "$-\frac{1}{2}$" appearing in the coefficient of $\frac{d}{2s}$
of the second bound of Proposition \ref{refined}, they are certainly
not the best possible. One might hope to do better with a closer
examination of the numerical functions arising in the proofs. We
decided not to push here this investigation further. The same remark
applies to the constant term $s^4$ appearing in the bound
(\ref{bound0}) of previous Proposition \ref{nSintegra}.

With reference to Notations \ref{secondnot}, recall that $s=s_1$ denotes the degree
of the component $D_1$ of $S$ of minimal degree, and $\pi_1$ its sectional genus.

\begin{proposition}\label{nSnonintegra} Let $X\subseteq \mathbb P^r$
be an integral projective variety of dimension $\dim X=3$. Let
$$
S=r_1D_1 + r_2D_2 +\dots + r_lD_l
$$
be an effective Cartier divisor of $X$ as in Notations
\ref{secondnot}. Let $C\subset S$ be a curve of $S$ of degree $d>0$
and arithmetic genus $p_a(C)$. We have
\begin{equation}\label{bound1}
p_a(C)\leq \frac{d^2}{2s}+\frac{d}{2s}(2\pi_1-2-s)+4(\deg S)^6.
\end{equation}
\end{proposition}

\smallskip
\begin{proposition}\label{refined} Let $X\subseteq \mathbb P^r$
be an integral projective variety of dimension $\dim X=3$. Let
$$
S=r_1D_1 + r_2D_2 +\dots + r_lD_l
$$
be an effective Cartier divisor of $X$ as in Notations
\ref{secondnot}. Let $C\subset S$ be a curve of $S$ of degree $d$
and arithmetic genus $p_a(C)$. Assume $d>30(\deg S)^7$.

\smallskip
1) If there exists a component $D_i$ of $S$ of minimal degree $s$
such that $C\cap D_i$ is a curve and $C$ and $C\cap D_i$ have the
same degree, then
$$
p_a(C) \leq G^*(S,d) \leq
\frac{d^2}{2s}+\frac{d}{2s}(2\pi_1-2-s)+s^4.
$$

\smallskip
2) Otherwise, we have
$$
p_a(C) \leq
\frac{d^2}{2s}+\frac{d}{2s}\left(2\pi_1-2-s-\frac{1}{2}\right).
$$
\end{proposition}

\smallskip
\begin{corollary}\label{intcompl} Let $X\subseteq \mathbb P^r$
be an integral projective variety of dimension $\dim X=3$. Let
$$
S=r_1D_1 + r_2D_2 +\dots + r_lD_l
$$
be an effective Cartier divisor of $X$ as in Notations
\ref{secondnot}. Assume there exists a curve of $S$ of degree $d$
with $d>30(\deg S)^7$, and arithmetic genus strictly greater  than
$\frac{d^2}{2s}+\frac{d}{2s}\left(2\pi_1-2-s-\frac{1}{2}\right)$.
Then
$$
G(S,d)=G^*(S,d).
$$
Moreover, if $C$ is a curve of $S$ of degree $d$ with $d>30(\deg
S)^7$, we have:
$$
p_a(C)=G(S,d)\quad \implies \quad C\in \mathcal C^*(S,d).
$$
\end{corollary}

\smallskip
\begin{remark}
By Remark \ref{Halphen}, $(i)$, if $d$ is a multiple of $s$, and
$d>30(\deg S)^7$, previous Corollary \ref{intcompl} applies.
Therefore, we have  $G(S,d)=G^*(S,d)$, and a curve of maximal genus
is necessarily contained in a component of $S$ of minimal degree.
\end{remark}

\smallskip
\begin{corollary}\label{CY} Let $X\subseteq \mathbb P^r$
be a smooth and subcanonical projective variety of dimension $\dim
X=3$. Let
$$
S=r_1D_1 + r_2D_2 +\dots + r_lD_l
$$
be an effective Cartier divisor of $X$ as in Notations
\ref{secondnot}. Assume that the integral components of $S$ of
minimal degree $s$ are smooth. Let $C\subset S$ be a curve of $S$ of
degree $d>30(\deg S)^7$ and arithmetic genus $p_a(C)$. We have
$$
p_a(C)\leq \frac{d^2}{2s}+\frac{d}{2s}(2\pi_1-2-s)+1,
$$
and the bound is sharp.
\end{corollary}

For the number $\epsilon_1$ and the bound  appearing in the
following Corollary \ref{P3} we refer to Remark \ref{Halphen},
$(iv)$. It is known that on certain surfaces the  bound is sharp. When $X=\mathbb P^3$,
Corollary \ref{P3} already appears in \cite[Proposition 4.3]{Liu}.
Our proof is somewhat different, since we need Bridgeland Theory
only to apply \cite[Lemma 4.1]{Liu}.

\smallskip
\begin{corollary}\label{P3} Assume $X\subset \mathbb P^4$. Let
$$
S=r_1D_1 + r_2D_2 +\dots + r_lD_l
$$
be an effective Cartier divisor of $X$ as in Notations
\ref{secondnot}. If $\deg X>1$ assume that $D_1$ is a hyperplane
section of $X$. Let $C\subset S$ be a curve of $S$ of degree
$d>30(\deg S)^7$ and arithmetic genus $p_a(C)$. We have
$$
p_a(C)\leq \frac{d^2}{2s}+\frac{d}{2}(s-4)+1-\epsilon_1,
$$
and the bound is sharp.
\end{corollary}

\smallskip
Now we are going to prove Proposition \ref{nSnonintegra} and
Proposition \ref{refined}. We start with Proposition
\ref{nSnonintegra}. We keep all the notations stated in Notations
\ref{secondnot}.

\smallskip
\begin{proof}[Proof of Proposition \ref{nSnonintegra}]
In view of Lemma \ref{planegenus} we may assume $s\geq 2$. We argue
by induction on $l$. Therefore, we first assume $l=1$. In this case
we have
$$
S=r_1D_1,\quad \deg S=r_1s.
$$
We examine this case  by induction on $r_1$. If $r_1=1$ our claim
follows from Proposition \ref{nSintegra}. Hence, we may assume
$r_1\geq 2$. Set:
$$
D=D_1,\quad E=(r_1-1)D.
$$
We have a natural exact sequence:
$$
0\to \frac{\mathcal I_C+\mathcal I_E}{\mathcal I_C}\to \mathcal
O_C\to \mathcal O_{C\cap E}\to 0.
$$
Let $C_1$ be the closed subscheme of $X$ whose ideal sheaf is
defined by
$$
\mathcal I_{C_1}=\mathcal I_{C}\, :\, \mathcal I_{E}.
$$
This means that {\it a local section $x$ of $\mathcal O_X$ belongs
to $\mathcal I_{C_1}$ if and only if $x\cdot f_1^{r_1-1}$ belongs to
$\mathcal I_{C}$}. Since  $f_1$ is a local generator of $\mathcal
I_D$,  $f_1^{r_1-1}$ is a local generator of $\mathcal I_E$, and $f_1^{r_1}$
is a local section of $\mathcal I_C$ because $C\subset S=r_1D$, we
have:
\begin{equation}\label{CDE}
\mathcal I_{C_1}\cdot\mathcal I_E={\mathcal I_{C}\cap\mathcal
I_E}\quad {\text{and}}\quad \mathcal I_{C}+\mathcal I_D\subseteq
{\mathcal I_{C_1}}.
\end{equation}
It follows:
$$
\mathcal O_{C_1}(-E)\cong \frac{\mathcal O_X(-E)}{\mathcal
I_{C_1}(-E)}\cong \frac{\mathcal I_E}{\mathcal I_{C_1}\cdot\mathcal
I_E}\cong \frac{\mathcal I_E}{\mathcal I_{C}\cap\mathcal I_E}\cong
\frac{\mathcal I_C+\mathcal I_E}{\mathcal I_C}.
$$
Therefore, we may rewrite the previous exact sequence:
\begin{equation}\label{bnHP}
0\to \mathcal O_{C_1}(-E)\to \mathcal O_C\to \mathcal O_{C\cap E}\to
0.
\end{equation}
Twisting with $\mathcal O_X(t)$, $t\in\mathbb Z$, we get the
following equality between Hilbert polynomials (recall that from (\ref{le}) we
are assuming that $E\sim a_1(r_1-1)H$):
\begin{equation}\label{nHP}
\chi \left(\mathcal O_{C}(t)\right)=\chi \left(\mathcal
O_{C_1}(t-a_1(r_1-1))\right)+\chi \left(\mathcal O_{C\cap E}(t)\right).
\end{equation}
Observe that $C\cap E$ is a curve, because
$C_{{\text{red}}}\subseteq C\cap E$.

If $\dim C_1=0$, from (\ref{nHP}) we deduce that $C$ and $C\cap E$
have the same degree, and that:
\begin{equation}\label{nnHP}
p_a(C)=p_a(C\cap E)-{\text{length}}(C_1)\leq p_a(C\cap E).
\end{equation}
In this case the bound (\ref{bound1}) for $p_a(C)$ follows by
induction applied to $C\cap E$.

If $\dim C_1=1$, denote by $d_1$ the degree of $C_1$, and by $d_2$
the degree of $C\cap E$. From (\ref{nHP}) we deduce
\begin{equation}\label{consHP}
d=d_1+d_2,\quad {\text{and}}\quad p_a(C)=p_a(C_1)+p_a(C\cap
E)+a_1(r_1-1)d_1-1.
\end{equation}
Moreover, since (compare with (\ref{CDE}))
$$
C_1\subseteq C\cap D\subseteq C\cap E,
$$
we also have $d_1\leq d_2$. Since $C_1\subset D$ and $C\cap E\subset
E$, by induction we get:
$$
p_a(C)=p_a(C_1)+p_a(C\cap E)+a_1(r_1-1)d_1-1
$$
$$
\leq\left[\frac{d_1^2}{2s}+\frac{d_1}{2s}(2\pi_1-2-s)+4s^6\right]+
\left[\frac{d_2^2}{2s}+\frac{d_2}{2s}(2\pi_1-2-s)+4(r_1-1)^6s^6\right]+a_1(r_1-1)d_1
$$
$$
=
\frac{d^2}{2s}+\frac{d}{2s}(2\pi_1-2-s)+a_1(r_1-1)d_1-\frac{1}{s}d_1d_2+4s^6+4(r_1-1)^6s^6
$$
$$
\leq
\frac{d^2}{2s}+\frac{d}{2s}(2\pi_1-2-s)+4r_1^6s^6+\left[a_1(r_1-1)d_1-\frac{1}{s}d_1d_2\right].
$$
Since $d=d_1+d_2$ and $d_1\leq d_2$, we have $d\leq 2d_2$. Hence, if
$d\geq 2a_1s(r_1-1)=2s^2(r_1-1)/n$, then
$$
a_1(r_1-1)d_1-\frac{1}{s}d_1d_2=d_1\left[a_1(r_1-1)-\frac{1}{s}d_2\right]\leq 0,
$$
and we get (\ref{bound1}). Otherwise, $$d< 2s^2(r_1-1)/n\leq
\sqrt{2\cdot 4(\deg S)^6}.$$ In this case (\ref{bound1}) follows
combining Lemma \ref{planegenus} with Lemma \ref{trecarte} (in
Lemma \ref{trecarte} put $\varphi= 4(\deg S)^6$).

\smallskip
This concludes the proof of the case $l=1$.

\smallskip
Now we are going to examine the case $l\geq 2$.

\smallskip
Set:
$$
D=r_1D_1+ r_2D_2+\dots+ r_{l-1}D_{l-1}, \quad E=r_lD_l.
$$
Define closed subschemes $C_D$ and $C_E$ of $X$ setting:
$$
\mathcal I_{C_D}=\mathcal I_C\,:\,\mathcal I_D \quad
{\text{and}}\quad \mathcal I_{C_E}=\mathcal I_C\,:\,\mathcal I_E.
$$
Similarly as in the case $l=1$, we have:
\begin{equation}\label{ideals}
\mathcal I_C+\mathcal I_D\subseteq \mathcal I_{C_E},\quad \mathcal
I_{C_E}\cdot\mathcal I_E={\mathcal I_{C}\cap\mathcal I_E},\quad
\mathcal I_C+\mathcal I_E\subseteq \mathcal I_{C_D},\quad \mathcal
I_{C_D}\cdot\mathcal I_D={\mathcal I_{C}\cap\mathcal I_D}.
\end{equation}
From which we obtain two exact sequences:
\begin{equation}\label{nsuccexact}
0\to \mathcal O_{C_E}(-E)\to \mathcal O_C\to \mathcal O_{C\cap E}\to
0\quad{\text{and}}\quad 0\to \mathcal O_{C_D}(-D)\to \mathcal O_C\to
\mathcal O_{C\cap D}\to 0.
\end{equation}

Now we are going to examine various cases.

\smallskip
If $\dim C_E=0$, from (\ref{nsuccexact}) we get $\dim C\cap E=1$ and
$\deg C \cap E=d$. From
 (\ref{nsuccexact}) and by induction we get:
$$
p_a(C)\leq p_a(C\cap E)\leq
\frac{d^2}{2s_l}+\frac{d}{2s_l}(2\pi_l-2-s_l)+4(\deg E)^6.
$$
If $s_1=s_l$ we have $\pi_1=\pi_l$ and we are done because $\deg
E\leq \deg S$. If $s_1<s_l$, using (\ref{stimapi}), we deduce:
$$
\frac{d^2}{2s_l}+\frac{d}{2s_l}(2\pi_l-2-s_l)+4(\deg E)^6
=\frac{d^2}{2s_1}+\frac{d}{2s_1}(2\pi_1-2-s_1)+4(\deg E)^6
$$
$$
+
d^2\frac{s_1-s_l}{2s_1s_l}+\frac{d}{2s_l}(2\pi_l-2-s_l)-\frac{d}{2s_1}(2\pi_1-2-s_1)
$$
$$
\leq \frac{d^2}{2s_1}+\frac{d}{2s_1}(2\pi_1-2-s_1)+4(\deg
S)^6+\left[d^2\frac{s_1-s_l}{2s_1s_l}+d(s_l+2)\right].
$$
If $d^2\frac{s_1-s_l}{2s_1s_l}+d(s_l+2)\leq 0$, we are done.
Otherwise we have
$$
d< \frac{2s_1s_l(s_l+2)}{s_l-s_1}\leq \sqrt{2\cdot 4(\deg S)^6}.
$$
And, combining Lemma \ref{planegenus} with Lemma \ref{trecarte}, we
get (\ref{bound1}). Here, in order to apply  Lemma \ref{trecarte},
we need the constant term in (\ref{bound1}) to be
approximately greater than $(\deg S)^6$.

\smallskip
If $\dim C_E=1$ and $\dim C\cap E=0$, then $\dim C\cap D=1$ and
$\deg C\cap D=d$. In fact, from (\ref{nsuccexact}) we have $\deg
C_E=d$. Since  $C_E\subseteq C\cap D$, it follows that  $C\cap D$ is
a curve of degree $d$. From (\ref{nsuccexact}) we deduce that $\dim
C_D=0$ and so, by induction,
$$
p_a(C)\leq p_a(C\cap D)\leq
\frac{d^2}{2s_1}+\frac{d}{2s_1}(2\pi_1-2-s_1)+4(\deg D)^6,
$$
and we are done because $\deg D\leq \deg S$.

\smallskip
If $\dim C_D=0$, then  $C\cap D$ is a curve of degree $d$. From
(\ref{nsuccexact}) and by induction we get
$$
p_a(C)\leq p_a(C\cap D)\leq
\frac{d^2}{2s_1}+\frac{d}{2s_1}(2\pi_1-2-s_1)+4(\deg D)^6,
$$
and we are done because $\deg D\leq \deg S$.

\smallskip
If $\dim C_D=1$ and $\dim C\cap D=0$, then $C\cap E$ is a curve of
degree $d$ because $C_D$ is a curve of degree $d$  contained in $C\cap E$.
Therefore, $\dim C_E=0$. We already examined this case
before.

\smallskip
Summing up, in view of previous analysis, we may assume that the
schemes  $C_E$, $C_D$, $C\cap E$ and $C\cap D$ all have dimension
$1$. We need the following claim.

\smallskip
{\bf {Claim}}. {\it Let $C^*$ be the closed subscheme of $X$
defined by the ideal $\mathcal I_{C^*}=\mathcal I_{C_E}\cap \mathcal
I_{C_D}$. Then $C^*$ is a curve contained in $C$, and
$$
\deg C_D+\deg C_E-\deg D\cdot\deg E\leq \deg C^*\leq d.
$$
}

\smallskip
{\it Proof of the Claim.} Since both $C_D$ and $C_E$ are contained
in $C^*$, we have $\dim C^*\geq 1$. On the other hand, by
(\ref{ideals}), we see that $C^*$ is contained in $C$. Therefore,
$C^*$ is a curve, and $\deg C^*\leq d$. In order to prove the
remaining inequality,  consider the homogeneous coordinate ring $R$
of  $X$, the homogeneous ideal $\mathbf a\subset R$ of ${C_E}$, the
homogeneous ideal $\mathbf b\subset R$ of $C_D$, and the Hilbert
polynomial $p_{C^*}(t)$ of $C^*$. For $t\gg 0$ we have:
$$
p_{C^*}(t)=\dim \frac{R_t}{(\mathbf a\cap\mathbf b)_t}=\dim
\frac{R_t}{\mathbf a_t\cap\mathbf b_t}
$$
$$
=\dim \frac{R_t}{\mathbf a_t}+\dim \frac{R_t}{\mathbf b_t}-\dim
\frac{R_t}{\mathbf a_t+\mathbf b_t}= \dim \frac{R_t}{\mathbf
a_t}+\dim \frac{R_t}{\mathbf b_t}-\dim \frac{R_t}{(\mathbf a+\mathbf
b)_t}.
$$
The ideal $\mathcal I_{C_E}+\mathcal I_{C_D}$ defines a subscheme
$C^{**}$ of $D\cap E$.
Recall  that $D$ and $E$ have no components in common (see Notations \ref{secondnot}), and that by (\ref{le}) we have
$$
D\sim (r_1a_1+\dots+r_{l-1}a_{l-1})H, \quad E\sim r_la_lH.
$$
It follows that
$D\cap E$ is a curve of degree $\deg (D\cap E)=(r_1a_1+\dots+r_{l-1}a_{l-1})r_la_ln$,
where $n$ is the degree of $X$. Since $\deg D\cdot\deg E=(r_1a_1+\dots+r_{l-1}a_{l-1})r_la_ln^2$, we get:
$$
\deg (D\cap E)\leq\deg D\cdot\deg E.
$$
Keeping this in mind, from the previous computation of the polynomial $p_{C^*}(t)$ we deduce
that, if $\dim C^{**}=0$, then  $\deg C^* = \deg C_D+\deg C_E$. And,
if $\dim C^{**}=1$, then  $$\deg C^*\geq \deg C_D+\deg C_E-\deg
D\cdot\deg E.$$

\smallskip
{\it This concludes the proof of the Claim.}

\smallskip
In view of the Claim we may assume that
$$
\min\{\deg C_D,\, \deg C_E\}\leq \frac{5}{9}d.
$$
In fact, otherwise, we have
$$
d< 9\deg D\cdot\deg E\leq \sqrt{2\cdot 4(\deg S)^6}.
$$
And, combining Lemma \ref{planegenus} with Lemma \ref{trecarte}, we
get (\ref{bound1}).

\smallskip
First we examine the case $\deg C_E\leq \frac{5}{9}d$.

Set $\deg C_E=d_1$ and $\deg C\cap E=d_2$. From (\ref{ideals}) we
see that $C_E\subset D$. By induction and (\ref{nsuccexact}) we get:
\begin{equation}\label{CE}
p_a(C)=p_a(C_E)+p_a(C\cap E)+r_la_l\deg C_E-1
\end{equation}
$$
\leq \left[\frac{d_1^2}{2s_1}+\frac{d_1}{2s_1}(2\pi_1-2-s_1)+4(\deg
D)^6\right]
$$
$$
+ \left[\frac{d_2^2}{2s_l}+\frac{d_2}{2s_l}(2\pi_l-2-s_l)+4(\deg
E)^6\right]+r_la_ld_1,
$$
with
$$
d=d_1+d_2,\quad d_1\leq \frac{5}{9}d.
$$

\smallskip
If $s_1=s_l$, then $\pi_1=\pi_l$. Hence, we have:
\begin{equation}\label{ce1}
p_a(C)\leq
\left[\frac{d^2}{2s_1}+\frac{d}{2s_1}(2\pi_1-2-s_1)+4(\deg
S)^6\right] +\left[r_la_ld_1-\frac{d_1d_2}{s_1}\right].
\end{equation}
Since $d_1\leq \frac{5}{9}d$, then $d_2\geq \frac{4}{9}d$. So,
if $d\geq \frac{9r_la_ls_1}{4}$, then
$r_la_ld_1-\frac{d_1d_2}{s_1}\leq 0$, and we are done. Otherwise
$$d\leq \frac{9r_la_ls_1}{4}=\frac{9r_ls_1^2}{4n}\leq \sqrt{2\cdot 4(\deg S)^6}.$$
Combining Lemma \ref{planegenus} with Lemma \ref{trecarte}, we get
(\ref{bound1}).

\smallskip
If $s_1<s_l$, set
$$
\psi=\psi(d_1)=\frac{d_1^2}{2s_1}+\frac{d_2^2}{2s_l}=\frac{d_1^2}{2s_1}+\frac{(d-d_1)^2}{2s_l}=
d_1^2\left[\frac{1}{2s_1}+\frac{1}{2s_l}\right]-\frac{dd_1}{s_l}+\frac{d^2}{2s_l}.
$$
The first derivative of $\psi$ only vanishes at
$d_1^*=\frac{s_1}{s_1+s_l}d$. Since $s_1<s_l$ we have
$d_1^*<\frac{5}{9}d$. By Lemma \ref{planegenus} and Lemma
\ref{trecarte} we may assume $d_1^*>1$. So we have:
$$
1<d_1^*<\frac{5}{9}d.
$$
In the range $\left[1,d_1^*\right]$ the function $\psi=\psi(d_1)$ is
decreasing, and in the range $\left[d_1^*,\frac{5}{9}d\right]$ is
increasing. Therefore, for every $d_1\in
\left[1,\frac{5}{9}d\right]$, we have
$$
\psi(d_1)\leq \max\left\{\psi(1),\,
\psi\left(\frac{5}{9}d\right)\right\}.
$$
Since $s_1<s_l$, an elementary computation proves that
$\psi\left(\frac{5}{9}d\right)\leq \frac{d^2}{3s_1}$.
Moreover, combining Lemma \ref{planegenus} with Lemma \ref{trecarte}, we
may assume that $ \left[\frac{1}{2s_1}+\frac{1}{2s_l}\right]-\frac{d}{s_l}\leq 0$.
So, we may assume $\psi(1)\leq \frac{d^2}{2s_l}$.
It follows that
$$
\psi(d_1)\leq \mu=\max\left\{\frac{d^2}{2s_l},\,
\frac{d^2}{3s_1}\right\}.
$$
We may write
$$
\mu=\frac{d^2}{2s_1}-yd^2
$$
with:
$$
{\text{either}}\quad y=\frac{s_l-s_1}{2s_1s_l}\quad{\text{or}}\quad y=\frac{1}{6s_1}.
$$
Taking into account (\ref{stimapi}), from (\ref{CE}) we get:
\begin{equation}\label{ce2}
p_a(C)\leq
\left[\frac{d_1^2}{2s_1}+\frac{d_1}{2s_1}(2\pi_1-2-s_1)+4(\deg
D)^6\right]
\end{equation}
$$
+ \left[\frac{d_2^2}{2s_l} +\frac{d_2}{2s_l}(2\pi_l-2-s_l)+4(\deg
E)^6\right]+r_la_ld_1
$$
$$
\leq \left[\frac{d^2}{2s_1}+\frac{d}{2s_1}(2\pi_1-2-s_1)+4(\deg
S)^6\right]
$$
$$
-yd^2+\frac{(d_1-d)}{2s_1}(2\pi_1-2-s_1) +ds_l+r_la_ld
$$
$$
\leq \left[\frac{d^2}{2s_1}+\frac{d}{2s_1}(2\pi_1-2-s_1)+4(\deg
S)^6\right] +\left[-yd^2+2d +ds_l+r_la_ld\right].
$$
If $\left[-yd^2+2d +ds_l+r_la_ld\right]\leq 0$,  we are done.
Otherwise
$$d<\frac{1}{y}(2+s_l+r_la_l)\leq \sqrt{2\cdot 4(\deg S)^6}.$$
Combining Lemma \ref{planegenus} with Lemma \ref{trecarte}, we get
(\ref{bound1}).

\smallskip
It remains to examine the last case, $\deg C_D\leq \frac{5}{9}d$.

Set $\deg C_D=d_1$ and $\deg C\cap D=d_2$. From (\ref{ideals}) we
see that $C_D\subset E$. By induction and (\ref{nsuccexact}) we get:
$$
p_a(C)=p_a(C_D)+p_a(C\cap D)+\alpha\deg C_D-1
$$
$$
\leq \left[\frac{d_1^2}{2s_l}+\frac{d_1}{2s_l}(2\pi_l-2-s_l)+4(\deg
E)^6\right]
$$
$$
+ \left[\frac{d_2^2}{2s_1}+\frac{d_2}{2s_1}(2\pi_1-2-s_1)+4(\deg
D)^6\right]+\alpha d_1,
$$
with
$$
d=d_1+d_2,\quad d_1\leq \frac{5}{9}d, \quad
\alpha=r_1a_1+\dots+r_{l-1}a_{l-1}.
$$

\smallskip
If $s_1=s_l$, then $\pi_1=\pi_l$. Hence, we have:
\begin{equation}\label{cd1}
p_a(C)\leq
\left[\frac{d^2}{2s_1}+\frac{d}{2s_1}(2\pi_1-2-s_1)+4(\deg
S)^6\right] +\left[\alpha d_1-\frac{d_1d_2}{s_1}\right].
\end{equation}
Since $d_1\leq \frac{5}{9}d$, we have $d_2\geq \frac{4}{9}d$. So,
if $d\geq \frac{9}{4}\alpha s_1$, then $\alpha
d_1-\frac{d_1d_2}{s_1}\leq 0$, and we are done. Otherwise, $d<
\frac{9}{4}\alpha s_1$. Combining Lemma \ref{planegenus} with Lemma
\ref{trecarte}, we get (\ref{bound1}).

\smallskip
If $s_1<s_l$, taking into account (\ref{stimapi}), and that
$$
\frac{d_1^2}{2s_l}+\frac{d_2^2}{2s_1}\leq
\frac{d_1^2}{2s_1}+\frac{d_2^2}{2s_1}=
\frac{d^2}{2s_1}-\frac{d_1d_2}{s_1},
$$
we have:
\begin{equation}\label{cd2}
p_a(C)\leq
\left[\frac{d_1^2}{2s_l}+\frac{d_1}{2s_l}(2\pi_l-2-s_l)+4(\deg
E)^6\right]
\end{equation}
$$+
\left[\frac{d_2^2}{2s_1}+\frac{d_2}{2s_1}(2\pi_1-2-s_1)+4(\deg
D)^6\right]+\alpha d_1
$$
$$
\leq
\left[\frac{d^2}{2s_1}+\frac{d_1+d_2}{2s_1}(2\pi_1-2-s_1)+4(\deg
S)^6\right]- \frac{d_1}{2s_1}(2\pi_1-2-s_1)
$$
$$
+\frac{d_1}{2s_l}(2\pi_l-2-s_l) -\frac{d_1d_2}{s_1}+\alpha d_1
$$
$$
\leq \left[\frac{d^2}{2s_1}+\frac{d}{2s_1}(2\pi_1-2-s_1)+4(\deg
S)^6\right] -\frac{d_1d_2}{s_1}+d_1s_l+2d_1 +\alpha d_1
$$
$$
\leq \left[\frac{d^2}{2s_1}+\frac{d}{2s_1}(2\pi_1-2-s_1)+4(\deg
S)^6\right]+d_1\left(-\frac{4}{9s_1}d+s_l+2+\alpha \right).
$$
If
$$
-\frac{4}{9s_1}d+s_l+2+\alpha\leq 0
$$
we are done. Otherwise, $d< \frac{9}{4}s_1(s_1+2+\alpha)$. Combining
Lemma \ref{planegenus} with Lemma \ref{trecarte}, we get
(\ref{bound1}).

\smallskip
This concludes the proof of Proposition \ref{nSnonintegra}.
\end{proof}

\smallskip
Next, we turn to the proof of Proposition \ref{refined}. We keep the
notations introduced in Notations \ref{secondnot}, and in the proof
of Proposition \ref{nSintegra}.

\smallskip
\begin{proof}[Proof of Proposition \ref{refined}.]
We argue by induction on $l$.

We first assume $l=1$. In this case we have
$$S=r_1D_1.$$
We examine this case by induction on $r_1$. If $r_1=1$ our claim
follows by Proposition \ref{nSintegra}. Hence, we may assume
$r_1\geq 2$. We keep the same notations we used in the proof of
Proposition \ref{nSnonintegra} in the case $l=1$ and $r_1\geq 2$.
Observe that, since $C_{red}\subset D\subset E=(r_1-1)D$, both $C\cap
D$ and $C\cap E$ are curves.

\smallskip
Assume $C$ and $C\cap E$ have the same degree, so $\dim C_1=0$. From
(\ref{nnHP}) we deduce $p_a(C)\leq p_a(C\cap E)$. If $C$ and $C\cap
D$ have the same degree, then also $C\cap E$ and $C\cap E\cap D
\,(=C\cap D)$ have the same degree, and the claim follows by
induction on $C\cap E$. If $C$ and $C\cap D$ do not have the same
degree, then also $C\cap E$ and $C\cap E\cap D$ do not have the same
degree, and again the claim follows by induction on $C\cap E$.

\smallskip
Assume now that $C$ and $C\cap E$ do not have the same degree (in
this case also $C$ and $C\cap D$ do not have the same degree). Then
$C_1$ is a curve, it is contained in $D$ (compare with (\ref{CDE})),
and by Proposition \ref{nSnonintegra} and (\ref{consHP}) we get:
$$
p_a(C)=p_a(C_1)+p_a(C\cap E)+a_1(r_1-1)d_1-1
$$
$$
\leq \left[\frac{d_1^2}{2s}+\frac{d_1}{2s}(2\pi_1-2-s)+4s^6\right]+
\left[\frac{d_2^2}{2s}+\frac{d_2}{2s}(2\pi_1-2-s)+4(r_1-1)^6s^6\right]+a_1(r_1-1)d_1
$$
$$
\leq
\frac{d^2}{2s}+\frac{d}{2s}(2\pi_1-2-s)+a_1(r_1-1)d_1-\frac{1}{s}d_1d_2+4r_1^6s^6
$$
$$
=\frac{d^2}{2s}+\frac{d}{2s}\left(2\pi_1-2-s-\frac{1}{2}\right)+\left[\frac{d}{4s}+a_1(r_1-1)d_1-\frac{1}{s}d_1d_2+4r_1^6s^6\right].
$$
In order to conclude the case $l=1$, we only have to prove that if
$d>30(\deg S)^7$ one has:
$$
\frac{d}{4s}+a_1(r_1-1)d_1-\frac{1}{s}d_1d_2+4r_1^6s^6\leq 0.
$$
To this purpose, recall that $d_2\geq \frac{d}{2}$. Therefore, for
$d>30(\deg S)^7$ we have:
$$
a_1(r_1-1)-\frac{1}{s}d_2<0.
$$
It follows that:
$$
a_1(r_1-1)d_1-\frac{1}{s}d_1d_2=d_1\left[a_1(r_1-1)-\frac{1}{s}d_2\right]\leq
a(r_1-1)-\frac{1}{s}d_2\leq a_1(r_1-1)-\frac{1}{2s}d.
$$
Hence, it suffices to prove that for $d>30(\deg S)^7$ one has:
$$
a_1(r_1-1)-\frac{1}{2s}d+4r_1^6s^6\leq -\frac{1}{4s}d,
$$
i.e.
$$
a_1(r_1-1)+4r_1^6s^6\leq \frac{1}{4s}d,
$$
which holds true for $d>30(\deg S)^7$.

\smallskip
This concludes the proof of Proposition \ref{refined} in the case
$l=1$.

\smallskip
Now, arguing by induction, we are going to examine  the case $l\geq
2$.

\smallskip
As in the proof of Proposition \ref{nSnonintegra}, set
$$
D=r_1D_1+r_2D_2+\dots+r_{l-1}D_{l-1}, \quad E=r_lD_l.
$$

\smallskip
First assume that $s_1=s_l$, i.e. all the components $D_i$ of $S$
have the same degree.

\smallskip
Assume there is a component $D_i$ of $S$ such that $C\cap D_i$ is a
curve of the same degree of $C$. We may assume $i=l$, i.e. that
$D_i=D_l$. Since $C\cap D_l\subseteq C\cap E\subseteq C$, it follows
that $C$ and $C\cap E$ are curves of the same degree. Using
(\ref{nsuccexact}) (compare with (\ref{bnHP}), (\ref{nHP}) and (\ref{nnHP}))
we get $p_a(C)\leq p_a(C\cap E)$. Then the claim
follows by induction applied to the curve $C\cap E$, which is
contained in $E$ and has the same degree of $C\cap E\cap D_l=C\cap
D_l$.

\smallskip
Now assume there is no such a component $D_i$. If $\dim C\cap E=0$,
then $C\cap D$ is a curve with the same degree of $C$, and using
(\ref{nsuccexact}) we have $p_a(C)\leq p_a(C\cap D)$. The claim
follows by induction on $C\cap D$ which is a curve contained in $D$.
If $C\cap E$ is a curve with the same degree of $C$, then
$p_a(C)\leq p_a(C\cap E)$, and the claim follows by induction on
$C\cap E$, which is a curve contained in $E$. Therefore, we may
assume that both $C\cap E$ and $C_E$ are curves. It follows that
also $C\cap D$ is a curve because $C_E$ is contained in $C\cap D$ by
(\ref{ideals}). If $C$ and $C\cap D$ have the same degree, then
$p_a(C)\leq p_a(C\cap D)$ and the claim follows by induction on
$C\cap D$. Therefore, we may assume that $C_E$, $C\cap E$, $C_D$,
$C\cap D$ are curves. In this case, as in the proof of Proposition
\ref{nSnonintegra}, we get (compare with (\ref{ce1}) and
(\ref{cd1})):
$$
p_a(C)\leq
\left[\frac{d^2}{2s_1}+\frac{d}{2s_1}(2\pi_1-2-s_1)+4(\deg
S)^6\right] +\left[\beta d_1-\frac{d_1d_2}{s_1}\right],
$$
where $\beta=r_la_l$ if $\deg C_E\leq \frac{5}{9}d$, and
$\beta=r_1a_1+\dots+r_{l-1}a_{l-1}$ if $\deg C_D\leq \frac{5}{9}d$.
Recall we also have $d=d_1+d_2$ and $d_1\leq \frac{5}{9}d$. We may
write:
$$
\left[\frac{d^2}{2s_1}+\frac{d}{2s_1}(2\pi_1-2-s_1)+4(\deg
S)^6\right] +\left[\beta d_1-\frac{d_1d_2}{s_1}\right]
$$
$$
=\left[\frac{d^2}{2s_1}+\frac{d}{2s_1}\left(2\pi_1-2-s_1-\frac{1}{2}\right)\right]
+\left[\frac{1}{4s_1}d+4(\deg S)^6+\beta
d_1-\frac{d_1d_2}{s_1}\right].
$$
In order to conclude the case $l\geq 2$ with $s_1=s_l$, we only have
to prove that if $d>30(\deg S)^7$ one has:
$$
\frac{1}{4s_1}d+4(\deg S)^6+\beta d_1-\frac{d_1d_2}{s_1}\leq 0.
$$
To this purpose, observe that $d_2\geq \frac{4}{9}d$. Therefore, for
$d>30(\deg S)^7$ we have:
$$
\beta-\frac{1}{s_1}d_2<0.
$$
It follows that:
$$
\beta
d_1-\frac{1}{s_1}d_1d_2=d_1\left(\beta-\frac{1}{s_1}d_2\right)\leq
\beta-\frac{1}{s_1}d_2\leq \beta-\frac{4}{9s_1}d.
$$
Hence, it suffices to prove that for $d>30(\deg S)^7$ we have:
$$
\frac{1}{4s_1}d+4(\deg S)^6+\beta-\frac{4}{9s_1}d\leq 0,
$$
i.e. that we have
$$
\beta+4(\deg S)^6\leq \frac{7}{36s_1}d,
$$
which holds true for $d>30(\deg S)^7$.

\smallskip
It remains to analyze the case $s_1<s_l$.

\smallskip
First assume there is a component $D_i$ of $S$ of  degree $s_1$ such
that $C\cap D_i$ is a curve and $C$ and $C\cap D_i$ have the same
degree. Then $C$ and $C\cap D$ are curves with the same degree. From
(\ref{nsuccexact}) we get $p_a(C)\leq p_a(C\cap D)$, and the claim
follows by induction.

\smallskip
Assume there is no such a component. If $C\cap E$ is not a curve,
then $C\cap D$ is a curve with the same degree of $C$, and
$p_a(C)\leq p_a(C\cap D)$ (compare with (\ref{ideals}) and (\ref{nsuccexact})). The claim follows by induction on $C\cap
D$. If $C\cap E$ is a curve with the same degree of $C$, then
$p_a(C)\leq p_a(C\cap E)$. From Proposition \ref{nSnonintegra}, we
know that
$$
p_a(C\cap E)\leq
\frac{d^2}{2s_l}+\frac{d}{2s_l}(2\pi_l-2-s_l)+4r_l^6s_l^6.
$$
Taking into account (\ref{stimapi}), and that $s_1<s_l$ and
$d>30(\deg S)^7$, a direct elementary computation proves that:
$$
\frac{d^2}{2s_l}+\frac{d}{2s_l}(2\pi_l-2-s_l)+4r_l^6s_l^6\leq
\frac{d^2}{2s_1}+\frac{d}{2s_1}\left(2\pi_1-2-s_1-\frac{1}{2}\right),
$$
and we are done. Hence, we may assume that $C\cap E$ and $C_E$ are
curves. It follows that also $C\cap D$ is a curve. If $C$ and $C\cap
D$ have the same degree, then $p_a(C)\leq p_a(C\cap D)$, and we are
done by induction. Therefore, we may assume that $C_E$, $C\cap E$,
$C_D$, $C\cap D$ are curves. In this case, if $\deg C_E\leq
\frac{5}{9}d$, as in the proof of Proposition \ref{nSnonintegra}, by
(\ref{ce2}) we have:
$$
p_a(C)\leq
\left[\frac{d^2}{2s_1}+\frac{d}{2s_1}(2\pi_1-2-s_1)+4(\deg
S)^6\right] +\left[-yd^2+2d +ds_l+r_la_ld\right],
$$
where:
$$
{\text{either}}\quad y=\frac{s_l-s_1}{2s_1s_l}\quad {\text{or}}\quad
y=\frac{1}{6s_1}.
$$
We may write
$$
\left[\frac{d^2}{2s_1}+\frac{d}{2s_1}(2\pi_1-2-s_1)+4(\deg
S)^6\right] +\left[-yd^2+2d +ds_l+r_la_ld\right]
$$
$$
=\left[\frac{d^2}{2s_1}+\frac{d}{2s_1}\left(2\pi_1-2-s_1-\frac{1}{2}\right)\right]+\left[\frac{d}{4s_1}+4(\deg
S)^6-yd^2+2d +ds_l+r_la_ld\right].
$$
We are done because for $d>30(\deg S)^7$ we have
$$
\frac{d}{4s_1}+4(\deg S)^6-yd^2+2d +ds_l+r_la_ld\leq 0.
$$
If  $\deg C_D\leq \frac{5}{9}d$, as in the proof of Proposition
\ref{nSnonintegra}, by (\ref{cd2}) we have:
$$
p_a(C)\leq
\left[\frac{d^2}{2s_1}+\frac{d}{2s_1}(2\pi_1-2-s_1)+4(\deg
S)^6\right]+d_1\left(-\frac{4}{9s_1}d+s_l+2+\alpha \right),
$$
where $d_1=\deg C_D$ and $\alpha=r_1a_1+\dots+r_{l-1}a_{l-1}$. We
may write
$$
\left[\frac{d^2}{2s_1}+\frac{d}{2s_1}(2\pi_1-2-s_1)+4(\deg
S)^6\right]+d_1\left(-\frac{4}{9s_1}d+s_l+2+\alpha \right)
$$
$$
=\left[\frac{d^2}{2s_1}+\frac{d}{2s_1}\left(2\pi_1-2-s_1-\frac{1}{2}\right)\right]+\left[\frac{d}{4s_1}+4(\deg
S)^6+d_1\left(-\frac{4}{9s_1}d+s_l+2+\alpha \right)\right].
$$
We are done because for $d>30(\deg S)^7$ we have
$$
\frac{d}{4s_1}+4(\deg S)^6+d_1\left(-\frac{4}{9s_1}d+s_l+2+\alpha
\right)\leq 0.
$$

\smallskip
This concludes the proof of Proposition \ref{refined}.
\end{proof}

\smallskip
\begin{proof}[Proof of Corollary \ref{intcompl}]
In view of Proposition \ref{refined}, we only have to prove that if
$C$ is a curve of maximal genus, then $C$ is {\it contained} in a
component of $S$ of minimal degree. By Proposition \ref{refined} and
by hypothesis, there is a component $D_i$ of minimal degree of $S$
such that $C$ and $C\cap D_i$ have the same degree. We may assume
that $i=1$. Now set:
$$
D=D_1\quad {\text{and}}\quad E=(r_1-1)D_1+ \dots +r_lD_l.
$$
As in the proof of Proposition \ref{nSnonintegra}, there is a
subscheme $C_D\subset C\cap E$ and  an exact sequence
$$
0\to \mathcal O_{C_D}(-D)\to \mathcal O_C\to \mathcal O_{C\cap D}\to
0.
$$
Since $C$ and $C\cap D$ have the same degree, we have $\dim C_D=0$,
and
$$
p_a(C)\leq p_a(C\cap D)-{\text{length}}(C_D)\leq p_a(C\cap D).
$$
Since $C$ has maximal genus, then $p_a(C)=p_a(C\cap D)$. Therefore,
$C_D$ is empty, and $C=C\cap D$.
\end{proof}

\section{Bound for the genus of a curve on a threefold.}

Let $X\subset \mathbb P^r$ be a projective factorial variety $X$ of
dimension $3$, with at worst isolated singularities, and Picard
group generated by the hyperplane section class. In this section
we are going to prove Theorem \ref{n3fold} stated in the Introduction, i.e.
we establish a Castelnuovo type bound for the arithmetic genus $p_a(C)$
of a projective scheme $C$ of dimension $1$ and degree $d\gg0$,
contained in $X$. Combining Theorem
\ref{n3fold} with Proposition \ref{refined} and Corollary
\ref{intcompl} from the previous section, we obtain the subsequent
corollaries Corollary \ref{cor1}, Corollary \ref{cor2}, and
Corollary \ref{ncor3} (the latter already stated in the Introduction). In short, we can prove that curves of $X$ of
maximum genus with respect to the degree must necessarily lie on a
hyperplane section. Corollary \ref{ncor3} appears to us to be a
significant step forward for the Castelnuovo's bound conjecture stated
in \cite[Conjecture 1.1]{Liu}. In Remark \ref{finalr}, $(iii)$, we
note that, in the case of integral curves, one can easily prove
Theorem \ref{n3fold} by combining  the main result in \cite{CCD} with
\cite[Lemma]{D3}.

Mutatis mutandis, the proof of Theorem \ref{n3fold} consists in
showing that the iterative method introduced by Liu in his work
\cite[Proof of Theorem 4.9.]{Liu} also fits our estimate
(\ref{nbound3fold}). In general, the estimate (\ref{nbound3fold}) is
finer than Liu's \cite[Theorem 4.9. and Corollary 4.11.]{Liu}. And,
in a certain sense, it is fine (see (\ref{boundintcompl}) in the Introduction,
and compare with Remark \ref{Halphen},
$(i)$, and Remark \ref{finalr}, $(ii)$).

Unlike what we saw in the previous section, where our properties
were proved with so-called classical methods, this Liu's iterative
method requires the use of some results based on Bridgeland
Stability Theory. We thought it useful to gather these results into
a single statement, Corollary \ref{Halphen2}, which requires two
preliminary lemmas, Lemma \ref{lemma:bogomolov} and Lemma
\ref{lemma:actualwalls}. All these properties appear in \cite{Liu},
and come from Bridgeland Stability Theory. The iterative method is
also based on other results of Liu, which in our opinion are
classical in nature, which we will simply recall when necessary
during the proof of Theorem \ref{n3fold}.

\begin{notations} $(i)$ In the sequel we denote by $\nu$ and $M(n)$
the same numbers appearing in the claim of Theorem \ref{n3fold}.

\smallskip
$(ii)$ Let $X$ be as in Theorem \ref{n3fold}. Fix an integer $d>0$.
Let $\mathcal C(X,d)$ ($\mathcal C^*(X,d)$ resp.) be the set of all
curves $C$ of degree $d$ contained in $X$ (in some hyperplane
section of $X$ resp.). Assume that $\mathcal C^*(X,d)$ is no empty.
Set
$$
G(X,d)=\max\left\{p_a(C)\,:\, C\in \mathcal C(S,d)\right\}\quad
{\text{and}}\quad G^*(X,d)=\max\left\{p_a(C)\,:\, C\in \mathcal
C^*(S,d)\right\}.
$$
\end{notations}

\smallskip
\begin{corollary}\label{cor1}
Let $X\subseteq\mathbb P^r$ be a projective factorial  variety of
dimension $3$, degree $n$, linear arithmetic genus $\pi_X$, with at
worst isolated singularities. Assume that the Picard group of $X$ is
generated by the hyperplane section class. Let $C\subset X$ be a
curve of degree $d>M^*(n)=6n^2M(n)^2$ and arithmetic genus
$p_a(C)$.

\smallskip
1) If there exists a hyperplane section $H$ of $X$ such that $C\cap
H$ is a curve and $C$ and $C\cap H$ have the same degree, then
$$
p_a(C) \leq \frac{d^2}{2n}+\frac{d}{2n}\left(2\pi_X-2-n\right)+n^4.
$$
\smallskip
2) Otherwise we have
$$
p_a(C) \leq
\frac{d^2}{2n}+\frac{d}{2n}\left(2\pi_X-2-n\right)-\frac{1}{n}\sqrt{d}.
$$
\end{corollary}

\smallskip
\begin{corollary}\label{cor2}
Let $X\subseteq\mathbb P^r$ be a projective factorial  variety of
dimension $3$, degree $n$, linear arithmetic genus $\pi_X$, with at
worst isolated singularities. Assume that the Picard group of $X$ is
generated by the hyperplane section class. Assume there exists a
curve of $X$ of degree $d$ with $d>M^*(n)=6n^2M(n)^2$, and
arithmetic genus strictly greater  than
$\frac{d^2}{2n}+\frac{d}{2n}\left(2\pi_X-2-n\right)-\frac{1}{n}\sqrt{d}$.
Then
$$
G(X,d)=G^*(X,d).
$$
Moreover, if $C$ is a curve of $X$ of degree $d$ with
$d>M^*(n)=6n^2M(n)^2$, we have:
$$
p_a(C)=G(X,d)\quad \implies \quad C\in \mathcal C^*(X,d).
$$
\end{corollary}

Before starting with the proof of Theorem \ref{n3fold}, we recall
some results coming from Bridgeland Stability Theory, which we need
in order to apply the iterative method of Liu (see Lemma \ref{lemma:bogomolov}, Lemma \ref{lemma:actualwalls},
and Corollary \ref{Halphen2} below). These results appear in Liu's article \cite{Liu}, cited below.
We refer to  \cite{Liu}  for more details (e.g. for the notion of semistability, actual wall, the category
$\mathcal A^b$, etc.).
In the sequel we denote by $\mathcal I_{C/\mathbb{P}^3}$ the ideal sheaf of a curve $C\subset \mathbb{P}^3$.

\begin{lemma}\label{lemma:bogomolov}
Let $C\subset \mathbb{P}^3$ be a curve of degree $d$.
Let $b_0$ be a real number such that $b_0<0$.
If
$\mathcal I_{C/\mathbb{P}^3}$  is $\sigma_{a,b_0}$-semistable for every real number $a>0$, then
\begin{equation*}
p_a(C)\le \dfrac{2}{-3b_0}d^2 + \left( - \dfrac{b_0}{3} - 2 \right)d
+ 1\ .
\end{equation*}
\end{lemma}
\begin{proof}
This is \cite[p.~15, Lemma 4.2]{Liu}. It follows applying  the
generalized Bogomolov-Gieseker inequality of Macr\`i in
$\mathbb{P}^3$ to $\mathcal I_{C/\mathbb{P}^3}$ with $(a,b) = (0, b_0)$
\cite{Macri'}.
\end{proof}

\begin{lemma}\label{lemma:actualwalls}
Let $C\subset \mathbb{P}^3$ be a Cohen-Macaulay curve of degree $d$.
Let $p$ be an integer such that
$$
0<\frac{4}{3}p\leq \sqrt{d}.
$$
Suppose that for every integer $k$ with $0<k\leq \frac{4}{3}p$, and for every surface
$S \in \vert \mathcal{O}_{\mathbb{P}^3}(k) \vert$, one has
either $\dim S\cap C=0$, or $\mathrm{dim}\ S\cap C = 1$
and
\begin{equation}\label{equation:smalldegree}
\deg(S\cap C) \leq k\sqrt{2d}-\dfrac{k^2}{2}.
\end{equation}
Then there is no actual wall for $\mathcal I_{C/\mathbb{P}^3} $ in the range
$a> 0$ and $ - \frac{4}{3}p \le b < 0$.
\end{lemma}

\begin{remark}
Since $k\leq \sqrt{d}$, from assumption (\ref{equation:smalldegree})
it follows that $\deg(S\cap C)<d$. Therefore, $C$ cannot be
contained in $S$.
\end{remark}

\begin{proof}[Proof of Lemma \ref{lemma:actualwalls}]
By \cite[Proposition 3.2]{Liu}, an actual wall for
$\mathcal I_{C/\mathbb{P}^3} $ in the range $a> 0$ and $ - \frac{4}{3}p \le b
< 0$ is given by an exact sequence of coherent sheaves in $\mathcal
A^b$:
\begin{equation*}
0 \rightarrow A \rightarrow \mathcal I_{C/\mathbb{P}^3}\rightarrow B
\rightarrow 0,
\end{equation*}
where we have two possibilities for $A$. One of the possibilities is
that $A \cong \mathcal{O}_{\mathbb{P}^3}(-S)$ is a line bundle, for
a divisor $S \in \vert \mathcal{O}_{\mathbb{P}^3}(k) \vert$ with
$1\leq k\le \frac{4}{3}p$. Since $A$ is a subsheaf of
$\mathcal I_{C/\mathbb{P}^3}$, the curve $C$ would be contained in $S$. This
is in contrast with our hypotheses.

The other possibility is that there exist one dimensional subschemes
$C_1,\,C_2\subset C$ and a divisor $S \in \vert
\mathcal{O}_{\mathbb{P}^3}(k) \vert$ such that $1\leq k\leq
\frac{4}{3}p$, $A\cong \mathcal I_{C_1/\mathbb{P}^3}(-S)$, $C_2\subset S\cap
C$, $d = \deg C_1+\deg C_2$, and
\begin{equation*}
d - \deg C_2 < \mathrm{min}\left\{  d - \dfrac{k^2}{2},\, d +
\dfrac{k^2}{2} - k\sqrt{2d}  \right\}.
\end{equation*}
Since $k\leq \sqrt{d}$, it follows that
$$
\mathrm{min}\left\{  d - \dfrac{k^2}{2},\, d + \dfrac{k^2}{2} -
k\sqrt{2d}  \right\}=d + \dfrac{k^2}{2} - k\sqrt{2d}.
$$
Therefore, we would have
$$
\deg S\cap C\geq \deg C_2 >k\sqrt{2d}-\dfrac{k^2}{2}.
$$
This is in contrast with our hypotheses.
\end{proof}

\begin{corollary}[Halphen's bound for non integral curves]\label{Halphen2}
Let $C\subset \mathbb{P}^3$ be a Cohen-Macaulay curve of degree $d$.
Let $p$ be an integer such that
$$
0<\frac{4}{3}p\leq \sqrt{d}.
$$
Suppose that for every integer $k$ with $0<k\leq \frac{4}{3}p$, and for every surface
$S \in \vert \mathcal{O}_{\mathbb{P}^3}(k) \vert$, one has
either $\dim S\cap C=0$, or $\mathrm{dim}\ S\cap C = 1$
and
\begin{equation*}
\deg(S\cap C) \leq k\sqrt{2d}-\dfrac{k^2}{2}.
\end{equation*}
Then one has:
\begin{equation}\label{equation:halphen}
p_a(C) \le\dfrac{d^2}{2p}+ \dfrac{2}{9}(2p-9) d + 1\ .
\end{equation}
\end{corollary}

\begin{remark}
Roughly speaking, we can state the hypotheses of the Corollary
\ref{Halphen2} as follows: {\it suppose that every surface of small
degree compared to $d$ intersects $C$ either in a $0$-dimensional scheme,
or in a curve
with small degree compared to $d$}. When $C$ is integral, this is
equivalent to say that $C$ is not contained in a surface of degree
small compared to $d$. Compare with Theorem \ref{HGP}.
\end{remark}

\begin{proof}[Proof of Corollary \ref{Halphen2}]
By Lemma \ref{lemma:actualwalls} there is no actual wall for $
\mathcal I_{C/\mathbb{P}^3} $ for any $a > 0$ and $b$ in the range
$-\frac{4}{3}p \leq b < 0$. On the other hand, $
\mathcal I_{C/\mathbb{P}^3}\in\mathcal A^b$ is $\sigma_{a,b}$-semistable for
any $a \gg 0 $ and $ b < 0 $ \cite[Lemma 2.3 and p. 10, two lines
above Proposition 3.2]{Liu}. It follows that $\mathcal I_{C/\mathbb{P}^3} $
is $\sigma_{a,b_0}$-semistable for any $a > 0$ and $b_0 =
-\frac{4}{3} p $. Inequality \eqref{equation:halphen} now follows
applying Lemma \ref{lemma:bogomolov}.\end{proof}

\smallskip
We are in position to prove Theorem \ref{n3fold}.

\smallskip
\begin{proof}[Proof of Theorem \ref{n3fold}] As in \cite[Proof of Theorem 4.9]{Liu}, the
proof consists in a finite sequence of steps.

\smallskip
{\bf{Step 0}}.

\smallskip
By \cite[Lemma 4.3]{Liu2} we may assume that $C$ is Cohen-Macaulay.
Set
$$
C_0=C, \quad d_0=d.
$$
By \cite[Section 3 and Lemma 3.14]{Liu}, we know that a generic
projection $\mathbb P^r\dasharrow\mathbb P^3$ restricts to a regular
map $X\to \mathbb P^3$ in such a way the image $C'_0$ of $C_0$ is a
Cohen-Macaulay curve, has the same degree of $C_0$, and
$p_a(C_0)\leq p_a(C'_0)$.

\smallskip
There are only three possibilities for $C'_0$.

\smallskip
$\bullet$ Type I.

\smallskip
For every surface $S \in \vert \mathcal{O}_{\mathbb{P}^3}(k) \vert$
with $1\leq k\leq \frac{4}{3}(n+1)$, one has either $\dim S\cap
C'_0=0$, or $\mathrm{dim}\ S\cap C'_0 = 1$ and
$$
\deg(S\cap C'_0) \leq k\sqrt{2d}-\dfrac{k^2}{2}.
$$
In view of  Corollary \ref{Halphen2} we deduce (set $p=n+1$):
$$
p_a(C'_0)\leq \frac{d^2}{2(n+1)}+\frac{2}{9}(2n-7)d+1.
$$
This number is $\leq \frac{d^2}{2n}+\frac{d}{2n}(2\pi_X-2-n)+4\nu^6$
if
$$
d>M_1=M_1(n)=\frac{8}{9}n(n+1)^2=\frac{2}{3}\nu(n+1).
$$
In this case we are done, because $M(n)>M_1(n)$.

\smallskip
$\bullet$ Type II.

\smallskip
$C'_0$ is contained in a surface  $S_0\subset \mathbb P^3$ with
$\deg S_0\in [1,\,\frac{4}{3}(n+1)]$. Taking the cone over $S_0$, we get a hypersurface
in $\mathbb P^r$ intersecting $X$ in a
surface $D_0\subset X$ containing $C_0$, with $\deg D_0\in
[n,\,\frac{4}{3}n(n+1)]$. Let
$s_0$ be the minimal degree of an integral component of $D_0$.
Observe that $n\leq s_0$ because $X$ is factorial and the Picard
group of $X$ is generated by the hyperplane section. From
Proposition \ref{nSnonintegra} we know that
$$
p_a(C_0)\leq \frac{d^2}{2s_0}+\frac{d}{2s_0}(2\pi_0-2-s_0)+4(\deg
D_0)^6.
$$
Here $\pi_0$ denotes the sectional genus of such a component of
minimal degree. A direct computation proves that this number is
$\leq \frac{d^2}{2n}+\frac{d}{2n}(2\pi_X-2-n)+4\nu^6$ if
$$
d>M_2=M_2(n)=\nu.
$$
In this case we are done, because $M(n)>M_2(n)$.

\smallskip
$\bullet$ Type III.

\smallskip
$C'_0$ is not of Type I and is not of Type II.

Therefore, there exists a surface  $S_1\subset \mathbb P^3$ with
$\deg S_1=k_1\in [1,\,\frac{4}{3}(n+1)]$,  such that $S_1\cap C'_0$
is a curve and such that
$$
\deg S_1\cap C'_0>  k_1\sqrt{2d}-\dfrac{{k_1}^2}{2}.
$$
Observe that, since $C'_0$ is Cohen-Macaulay,  we have $\deg
C'_0>\deg S_1\cap C'_0$. In fact, otherwise,  $S_1\cap C'_0=C'_0$,
and $C'_0$ would be contained in $S_1$, i.e. $C'_0$ would be of Type
II. In view of \cite[Lemma 3.1]{Liu}, we know there exists a closed subscheme ${C'_0}^*\subseteq C'_0$ and an exact sequence
$$
0\to \mathcal O_{{C'_0}^*}(-S_1)\to \mathcal O_{{C'_0}}\to \mathcal O_{C'_0\cap S_1}\to 0.
$$
Since $\deg C'_0>\deg S_1\cap C'_0$, it follows that $\dim {C'_0}^*=1$. Therefore, by
\cite[Proposition 3.15 and proof]{Liu}, we deduce the existence of a  curve $C_1\subset C_0$, a
surface $D_1\subset X$  with $\deg D_1=k_1n\in
[n,\,\frac{4}{3}n(n+1)]$, such that
\begin{equation}\label{stimaC_0}
p_a(C_0)=p_a(C_1)+p_a(D_1\cap C_0)+k_1\deg C_1-1,
\end{equation}
$$
\deg (D_1\cap C_0) >  k_1\sqrt{2d}-\dfrac{{k_1}^2}{2},
$$
and
$$
\deg C_1+\deg(D_1\cap C_0)=d=\deg C_0.
$$
Set
$$
d_1=\deg C_1,\quad d'_0=d_0-d_1=d-d_1=\deg (D_1\cap C_0).
$$
Since we are interested to bound $p_a(C_0)$ from above, in the
sequel we may assume that $C_1$ is Cohen-Macaulay \cite[Lemma
4.3]{Liu2}. As for $C_0$, denote by $C'_1\subset \mathbb P^3$ a generic
projection of $C_1$. Set:
$$
M_3=4\nu^6.
$$
Notice that $M_3>M_1>M_2$.

\smallskip
Now, we are going to prove that in the following two subcases

\smallskip
1) $d_1>M_3$ and $C'_1$ is of Type I or II,

\smallskip
2) $d_1\leq M_3$,

the inequality (\ref{nbound3fold}) holds when $d>M(n)$.

\smallskip
To this aim, first we notice that, since
$d'_0=d-d_1>k_1\sqrt{2d}-\frac{k_1^2}{2}$, we have
\begin{equation}\label{primod}
d>\frac{9}{8}\nu^2\quad \implies \quad d'_0>M_2=\nu.
\end{equation}
Hence, as in Type II, we have
$$
p_a(C_0\cap D_1)\leq
\frac{{d'_0}^2}{2n}+\frac{d'_0}{2n}(2\pi_X-2-n)+4\nu^6.
$$
Set
$$
M_4=\frac{9}{8}\nu^2,
$$
and notice that $M_3>M_4$. From (\ref{stimaC_0}) we have:
$$
p_a(C_0)=p_a(C_1)+p_a(C_0\cap D_1)+k_1d_1-1
$$
$$
\leq
\frac{{d'_0}^2}{2n}+\frac{d'_0}{2n}(2\pi_X-2-n)+4\nu^6+[p_a(C_1)+k_1d_1]
$$
$$
=\frac{({d-d_1})^2}{2n}+\frac{d-d_1}{2n}(2\pi_X-2-n)+4\nu^6+[p_a(C_1)+k_1d_1]
$$
$$
=\frac{{d}^2}{2n}+\frac{d}{2n}(2\pi_X-2-n)+4\nu^6+\left[-\frac{dd_1}{n}+\frac{d_1^2}{2n}-\frac{d_1}{2n}(2\pi_X-2-n)+p_a(C_1)+k_1d_1\right].
$$
Set:
$$
R_1=-\frac{dd_1}{n}+\frac{d_1^2}{2n}-\frac{d_1}{2n}(2\pi_X-2-n)+p_a(C_1)+k_1d_1.
$$
In order to prove (\ref{nbound3fold}), it suffices to prove that
$R_1\leq 0$.

\smallskip
In the first case 1) we have $d_1>M_3$ and $C'_1$ is of  Type I or
II. As in Type I or II we have:
$$
p_a(C_1)\leq \frac{d_1^2}{2n}+\frac{d_1}{2n}(2\pi_X-2-n)+4\nu^6.
$$
Hence we have:
$$
R_1\leq -\frac{dd_1}{n}+\frac{d_1^2}{n}+4\nu^6+k_1d_1.
$$
In order to conclude the case 1), it suffices that
\begin{equation}\label{nstruttura}
-\frac{d_1(d-d_1)}{n}+4\nu^6+k_1d_1\leq 0.
\end{equation}
Since $d_1> M_3=4\nu^6$, dividing by $d_1$, it suffices that:
\begin{equation*}
-\frac{d-d_1}{n}+1+k_1\leq 0,
\end{equation*}
i.e.
$$
{d-d_1}\geq n({1}+k_1).
$$
As in (\ref{primod}) one sees that
$$
d>\frac{9}{8}\nu^2\quad \implies \quad d'_0>n({1}+k_1).
$$
This concludes the analysis of the case 1).

\smallskip
Now we are going to examine the case 2). We have $d_1\leq M_3$.
Taking into account Lemma \ref{planegenus}, we may write:
$$
R_1=-\frac{dd_1}{n}+\frac{d_1^2}{2n}-\frac{d_1}{2n}(2\pi_X-2-n)+p_a(C_1)+k_1d_1
$$
$$
=-\frac{d'_0d_1}{n}-\frac{d_1^2}{2n}-\frac{d_1}{2n}(2\pi_X-2-n)+p_a(C_1)+k_1d_1
$$
$$
\leq
-\frac{d'_0d_1}{n}-\frac{d_1^2}{2n}+2d_1+\frac{1}{2}(d_1-1)(d_1-2)+k_1d_1.
$$
Dividing by $d_1$, we see that in order to prove that $R_1\leq 0$ it
suffices that:
$$
\frac{d-d_1}{n}\geq -\frac{d_1}{2n}+2+\frac{1}{2}d_1+k_1.
$$
This holds true because $d>M(n)$ and $d_1\leq M_3$.

\smallskip
Summing up, based on the previous analysis, we have proved the
inequality (\ref{nbound3fold}), except in the case
$$
d_1>M_3\quad{\text{and}}\quad C'_1\quad{\text{is of Type III.}}
$$
In this case, substituting $C_0$ with $C_1$, we repeat the same
argument on $C_1$.

\smallskip
{\bf{Step 1}}.

\smallskip
As before, we find a curve $C_2\subset C_1$ of degree $d_2$, and a
surface $D_2\subset X$ with $\deg D_2=k_2n\in
[n,\,\frac{4}{3}n(n+1)]$ such that
\begin{equation}\label{stimaC_1}
p_a(C_1)=p_a(C_2)+p_a(D_2\cap C_1)+k_2d_2-1,
\end{equation}
$$
\deg (D_2\cap C_1) >  k_2\sqrt{2d}-\dfrac{{k_2}^2}{2},
$$
and
$$
\deg C_2+\deg(D_2\cap C_1)=d_1=\deg C_1.
$$
Set $\deg (D_2\cap C_1)=d'_1=d_1-d_2$, and denote by $C'_2$ a generic
projection of $C_2$.

\smallskip
We are going to prove that in the following two subcases

\smallskip
1) $d_2>M_3$ and $C'_2$ is of Type I or II,

\smallskip
2) $d_2\leq M_3$,

the inequality (\ref{nbound3fold}) holds when $d>M(n)$.

\smallskip
Combining (\ref{stimaC_0}) with (\ref{stimaC_1}), we have:
$$
p_a(C)=p_a(C_2)+p_a(C_1\cap D_2)+p_a(C_0\cap D_1)+k_2d_2+k_1d_1-2.
$$
As in (\ref{primod}), since $d_1>M_3>M_4$, we have $d'_1>M_2$.
Therefore, as in Type II, we have:
$$
p_a(C_1\cap D_2)\leq
\frac{{d'_1}^2}{2n}+\frac{d'_1}{2n}(2\pi_X-2-n)+4\nu^6.
$$
Hence:
$$
p_a(C)=\sum_{i=0}^{1}p_a(C_i\cap
D_{i+1})+p_a(C_2)+\left(\sum_{j=1}^{2}k_jd_j\right)-2
$$
$$
\leq
\left[\sum_{i=0}^{1}\left(\frac{{d'_i}^2}{2n}+\frac{d'_i}{2n}(2\pi_X-2-n)+4\nu^6\right)\right]+p_a(C_2)+\left(\sum_{j=1}^{2}k_jd_j\right)-2
$$
$$
=\frac{1}{2n}\left(\sum_{i=0}^{1}{d'_i}\right)^2
-\frac{1}{n}\left(\sum_{0\leq x<y\leq 1}d'_xd'_y\right)
$$
$$
+\frac{1}{2n}(2\pi_X-2-n)\left(\sum_{i=0}^{1}d'_i\right)+2(4\nu^6)+p_a(C_2)+\left(\sum_{j=1}^{2}k_jd_j\right)-2
$$
$$
=\frac{1}{2n}(d-d_2)^2+\frac{d-d_2}{2n}(2\pi_X-2-n)+p_a(C_2)+2(4\nu^6)
$$
$$
+\left(\sum_{j=1}^{2}k_jd_j\right)-2-\frac{1}{n}\left(\sum_{0\leq
x<y\leq 1}d'_xd'_y\right)
$$
$$
=\frac{{d}^2}{2n}+\frac{d}{2n}(2\pi_X-2-n)+4\nu^6+R_2,
$$
where
$$
R_2=-\frac{dd_2}{n}+\frac{d_2^2}{2n}-\frac{d_2}{2n}(2\pi_X-2-n)+4\nu^6+p_a(C_2)
$$
$$
+\left(\sum_{j=1}^{2}k_jd_j\right)-2-\frac{1}{n}\left(\sum_{0\leq
x<y\leq 1}d'_xd'_y\right).
$$
Therefore, in order to prove (\ref{nbound3fold}) it suffices to prove
that $R_2\leq 0$.

\smallskip
In the case 1) we have
$$
p_a(C_2)\leq \frac{{d_2}^2}{2n}+\frac{d_2}{2n}(2\pi_X-2-n)+4\nu^6.
$$
Inserting in $R_2$ we get:
$$
R_2\leq
-\frac{dd_2}{n}+\frac{d_2^2}{n}+2(4\nu^6)+\left(\sum_{j=1}^{2}k_jd_j\right)-2-\frac{1}{n}\left(\sum_{0\leq
x<y\leq 1}d'_xd'_y\right)
$$
$$
=-\frac{d_2(d_2+\sum_{i=0}^{1}d'_i)}{n}+\frac{d_2^2}{n}+2(4\nu^6)+\left(\sum_{j=1}^{2}k_jd_j\right)-2-\frac{1}{n}\left(\sum_{0\leq
x<y\leq 1}d'_xd'_y\right)
$$
$$
\leq
-\frac{\sum_{i=0}^{1}d_2d'_i}{n}+2(4\nu^6)+\left(\sum_{j=1}^{2}k_jd_j\right)-\frac{1}{n}\left(\sum_{0\leq
x<y\leq 1}d'_xd'_y\right)
$$
$$
=-\frac{1}{n}\sum_{i=0}^{1}d'_id_{i+1}+2(4\nu^6)+\left(\sum_{j=1}^{2}k_jd_j\right)
$$
\begin{equation}\label{both}
=\sum_{i=0}^{1}\left[-\frac{1}{n}d'_id_{i+1}+k_{i+1}d_{i+1}+4\nu^6\right].
\end{equation}
Since $d=d_0$, $d_1$ and $d_2$ are $>M_3$, as in (\ref{nstruttura})
one sees that both the summands in (\ref{both}) are $\leq 0$. This
concludes the analysis of the case 1).

\smallskip
In the second case 2), we have $d_2\leq M_3$. Now, we may write:
$$
R_2=-\frac{dd_2}{n}+\frac{d_2^2}{2n}-\frac{d_2}{2n}(2\pi_X-2-n)+4\nu^6+p_a(C_2)
$$
$$
+\left(\sum_{j=1}^{2}k_jd_j\right)-2-\frac{1}{n}\left(\sum_{0\leq
x<y\leq 1}d'_xd'_y\right)
$$
$$
=-\frac{1}{n}d_2\left(\sum_{i=0}^{1}d'_i\right)-\frac{d_2^2}{2n}-\frac{d_2}{2n}(2\pi_X-2-n)+4\nu^6
$$
$$
+p_a(C_2)+\left(\sum_{j=1}^{2}k_jd_j\right)-2-\frac{1}{n}\left(\sum_{0\leq
x<y\leq 1}d'_xd'_y\right)
$$
$$
=\sum_{j=1}^{2}\left(-\frac{d_jd'_{j-1}}{n}+k_jd_j\right)
-\frac{d_2^2}{2n}-\frac{d_2}{2n}(2\pi_X-2-n)+4\nu^6 +p_a(C_2)-2.
$$
As in (\ref{nstruttura}), since $d_1>M_3$, one sees that
$$
-\frac{d_2d'_{1}}{n}+k_2d_2\leq 0.
$$
Hence, using also Lemma \ref{planegenus} and (\ref{stimapi}), we have:
$$
R_2\leq
-\frac{d_1d'_{0}}{n}+k_1d_1-\frac{d_2^2}{2n}-\frac{d_2}{2n}(2\pi_X-2-n)+4\nu^6
+p_a(C_2)
$$
$$
\leq
-\frac{d_1d'_{0}}{n}+k_1d_1-\frac{d_2^2}{2n}+2{d_2}+4\nu^6+\frac{1}{2}(d_2-1)(d_2-2).
$$
Since $d_1>M_3$, $d_2\leq M_3$, and $d_2<d_1$, dividing by $d_1$ one sees that in
order to prove that $R_2\leq 0$ it suffices that:
$$
\frac{1}{n}(d-d_1)\geq k_1+3+\frac{d_2}{2}.
$$
On the other hand, we know that $d-d_1>
k_1\sqrt{2d}-\frac{k_1^2}{2}$. Therefore,  $R_2\leq 0$ when
$$
\sqrt{2d}\geq n\left(\frac{k_1}{2}+4+\frac{M_3}{2}\right).
$$
This holds true because $d>M(n)$.

\smallskip
Based on the previous analysis, we have proved the inequality
(\ref{nbound3fold}), except in the case
$$
d_2>M_3\quad{\text{and}}\quad C'_2\quad{\text{is of Type III.}}\quad
$$
In this case, substituting $C_1$ with $C_2$, we repeat the same
argument on $C_2$.

\smallskip
Continuing in this way, after a certain number $m-1\,(\geq 1)$ of
steps, either we manage to prove (\ref{nbound3fold}), or we will find
ourselves in the situation described below (here, as for the choice of the letter "$m$",
we decide to keep the same notation as in \cite[p. 23]{Liu}; this letter  "$m$" should not be confused with the one defined above
in the claim of Theorem \ref{HGP}, and in (\ref{divide})).

\smallskip
{\bf{Step m-1}}.

\smallskip
We find two sequences of curves
$$C'_m\subset C'_{m-1}\subset \dots\subset C'_2\subset C'_1\subset C'_0=C'\subset \mathbb P^3,$$
$$C_m\subset C_{m-1}\subset \dots\subset C_2\subset C_1\subset C_0=C\subset X,$$
and a sequence of surfaces of $X$
$$
D_m, \, D_{m-1},\,\dots,\,D_1
$$
such that $\deg C'_i=\deg C_i=d_i$ for every $i=0,\dots,m$,
\begin{equation}\label{degrees}
0<d_m<d_{m-1}<\dots < d_1<d_0=d,
\end{equation}
$$
d_i>M_3\quad \forall \, 0\leq i\leq m-1,
$$
$$
\deg D_i=k_in\in [n,\,\frac{4}{3}n(n+1)]\quad \forall \, 1\leq i\leq
m,
$$
$$
\deg C_{i+1}+\deg(D_{i+1}\cap C_i)=d_i \quad\forall \, 0\leq i\leq
m-1,
$$
$$
d'_i=d_i-d_{i+1}> k_{i+1}\sqrt{2d}-\frac{k_{i+1}^2}{2}\quad \forall
\, 0\leq i\leq m-1,
$$
\begin{equation}\label{stimaC_m}
p_a(C)=\sum_{i=0}^{m-1}p_a(C_i\cap
D_{i+1})+p_a(C_m)+\left(\sum_{j=1}^{m}k_jd_j\right)-m.
\end{equation}

\smallskip
Now we are going  to prove that, regardless of how many steps are
taken, in the following two subcases

\smallskip
1) $d_m>M_3$ and $C'_m$ is of Type I or II,

\smallskip
2) $d_m\leq M_3$,

we can prove the inequality (\ref{nbound3fold}) for $d>M(n)$. This
will conclude the proof of  Theorem \ref{n3fold}, because, since the
sequence (\ref{degrees}) of degrees is strictly decreasing, at a
certain point $d_m\leq M_3$ must be.

\smallskip
We are going to examine the two subcases 1) and 2).

As in (\ref{primod}), since every $d_i>M_3>M_4$ ($0\leq i\leq m-1$),
we have $d'_i>M_2$. Therefore, as  in Type II, for every  $0\leq
i\leq m-1$ we have:
$$
p_a(C_i\cap D_{i+1})\leq
\frac{{d'_i}^2}{2n}+\frac{d'_i}{2n}(2\pi_X-2-n)+4\nu^6.
$$
From (\ref{stimaC_m}) we deduce:
$$
p_a(C)=\sum_{i=0}^{m-1}p_a(C_i\cap
D_{i+1})+p_a(C_m)+\left(\sum_{j=1}^{m}k_jd_j\right)-m
$$
$$
\leq
\left[\sum_{i=0}^{m-1}\left(\frac{{d'_i}^2}{2n}+\frac{d'_i}{2n}(2\pi_X-2-n)+4\nu^6\right)\right]+p_a(C_m)+\left(\sum_{j=1}^{m}k_jd_j\right)-m
$$
$$
=\frac{1}{2n}\left(\sum_{i=0}^{m-1}{d'_i}\right)^2
-\frac{1}{n}\left(\sum_{0\leq x<y\leq m-1}d'_xd'_y\right)
$$
$$
+\frac{1}{2n}(2\pi_X-2-n)\left(\sum_{i=0}^{m-1}d'_i\right)+m(4\nu^6)+p_a(C_m)+\left(\sum_{j=1}^{m}k_jd_j\right)-m
$$
$$
=\frac{1}{2n}(d-d_m)^2+\frac{d-d_m}{2n}(2\pi_X-2-n)+p_a(C_m)+m(4\nu^6)
$$
$$
+\left(\sum_{j=1}^{m}k_jd_j\right)-m-\frac{1}{n}\left(\sum_{0\leq
x<y\leq m-1}d'_xd'_y\right)
$$
$$
=\frac{{d}^2}{2n}+\frac{d}{2n}(2\pi_X-2-n)+4\nu^6 + R_m,
$$
where
$$
R_m=-\frac{dd_m}{n}+\frac{d_m^2}{2n}-\frac{d_m}{2n}(2\pi_X-2-n)+(m-1)4\nu^6+p_a(C_m)
$$
$$
+\left(\sum_{j=1}^{m}k_jd_j\right)-m-\frac{1}{n}\left(\sum_{0\leq
x<y\leq m-1}d'_xd'_y\right).
$$
We only have to prove that $R_m\leq 0$.

\smallskip
In the case 1), since $d_m>M_3$, we have
$$
p_a(C_m)\leq \frac{{d_m}^2}{2n}+\frac{d_m}{2n}(2\pi_X-2-n)+4\nu^6.
$$
Inserting in $R_m$ we get:
$$
R_m\leq
-\frac{dd_m}{n}+\frac{d_m^2}{n}+m(4\nu^6)+\left(\sum_{j=1}^{m}k_jd_j\right)-m-\frac{1}{n}\left(\sum_{0\leq
x<y\leq m-1}d'_xd'_y\right)
$$
$$
=-\frac{d_m(d_m+\sum_{i=0}^{m-1}d'_i)}{n}+\frac{d_m^2}{n}+m(4\nu^6)+\left(\sum_{j=1}^{m}k_jd_j\right)-m-\frac{1}{n}\left(\sum_{0\leq
x<y\leq m-1}d'_xd'_y\right)
$$
$$
\leq
-\frac{\sum_{i=0}^{m-1}d_md'_i}{n}+m(4\nu^6)+\left(\sum_{j=1}^{m}k_jd_j\right)-\frac{1}{n}\left(\sum_{0\leq
x<y\leq m-1}d'_xd'_y\right)
$$
$$
=-\frac{1}{n}\sum_{i=0}^{m-1}d'_i\left(d'_{i+1}+\dots+d'_{m-1}+d_m\right)+m(4\nu^6)+\left(\sum_{j=1}^{m}k_jd_j\right)
$$
$$
=\sum_{i=0}^{m-1}\left[-\frac{1}{n}d'_id_{i+1}+k_{i+1}d_{i+1}+4\nu^6\right].
$$
Since all the degrees $d=d_0,d_1,\dots,d_m$ are $>M_3$, as in
(\ref{nstruttura}) we see that each of the $m$ addends of the
summation are $\leq 0$. This concludes the analysis of the case 1).

\smallskip
Now, we are going to examine the second subcase 2),  when $d_m\leq
M_3$. We have:
$$
R_m=-\frac{dd_m}{n}+\frac{d_m^2}{2n}-\frac{d_m}{2n}(2\pi_X-2-n)+(m-1)4\nu^6+p_a(C_m)
$$
$$
+\left(\sum_{j=1}^{m}k_jd_j\right)-m-\frac{1}{n}\left(\sum_{0\leq
x<y\leq m-1}d'_xd'_y\right)
$$
$$
=-\frac{1}{n}d_m\left(\sum_{i=0}^{m-1}d'_i\right)-\frac{d_m^2}{2n}-\frac{d_m}{2n}(2\pi_X-2-n)+(m-1)4\nu^6
$$
$$
+p_a(C_m)+\left(\sum_{j=1}^{m}k_jd_j\right)-m-\frac{1}{n}\left(\sum_{0\leq
x<y\leq m-1}d'_xd'_y\right)
$$
$$
=\sum_{j=1}^{m}\left(-\frac{d_jd'_{j-1}}{n}+k_jd_j\right)
-\frac{d_m^2}{2n}-\frac{d_m}{2n}(2\pi_X-2-n)+(m-1)4\nu^6 +p_a(C_m)-m
$$
$$
=\left[\sum_{j=2}^{m-1}\left(-\frac{d_jd'_{j-1}}{n}+k_jd_j+4\nu^6\right)\right]+
\left(-\frac{d_md'_{m-1}}{n}+k_md_m\right)+
\left(-\frac{d_1d'_{0}}{n}+k_1d_1+4\nu^6\right)
$$
$$
-\frac{d_m^2}{2n}-\frac{d_m}{2n}(2\pi_X-2-n) +p_a(C_m)-m.
$$

\smallskip
As in (\ref{nstruttura}), since each $d_j$, $0\leq j\leq {m-1}$, is
$>M_3$, one sees that
$$
-\frac{d_jd'_{j-1}}{n}+k_jd_j+4\nu^6\leq 0\quad \forall\, 2\leq
j\leq {m-1},
$$
and
$$
-\frac{d_md'_{m-1}}{n}+k_md_m\leq 0.
$$
Therefore, taking into account Lemma \ref{planegenus} and
(\ref{stimapi}), we have:
$$
R_m\leq
\left(-\frac{d_1d'_{0}}{n}+k_1d_1+4\nu^6\right)-\frac{d_m^2}{2n}-\frac{d_m}{2n}(2\pi_X-2-n)
+p_a(C_m)
$$
$$
\leq
\left(-\frac{d_1d'_{0}}{n}+k_1d_1+4\nu^6\right)-\frac{d_m^2}{2n}+2{d_m}
+\frac{1}{2}(d_m-1)(d_m-2).
$$
Since $d_1>M_3$ and $d_m\leq M_3$, dividing by $d_1$ one sees that
in order to prove that $R_m\leq 0$ it suffices that
$$
\frac{1}{n}(d-d_1)\geq k_1+3+\frac{d_m}{2}.
$$
On the other hand, we know that
$$
d-d_{1}> k_{1}\sqrt{2d}-\frac{k_{1}^2}{2}.
$$
So, it suffices that:
$$
\sqrt{2d}\geq n\left(\frac{k_1}{2}+4+\frac{M_3}{2}\right),
$$
which holds true because  $d>M(n)$. Since $1\leq k_1\leq \frac{4}{3}(n+1)$ and $M_3=4\nu^6$,
this explains our definition of $M(n)$.

\smallskip
This concludes the proof of Theorem \ref{n3fold}.
\end{proof}

\smallskip
\begin{proof}[Proof of Corollary \ref{cor1}]
Set
$$
C_0=C, \quad d_0=d.
$$
Let  $C'_0\subset \mathbb P^3$ be a generic projection  of $C_0$. We
have $\deg C_0=\deg C'_0$, $p_a(C_0)\leq p_a(C'_0)$. By \cite[Lemma
4.3 and proof]{Liu2} we may assume $C_0$ is Cohen-Macaulay. So also
$C'_0$ is Cohen-Macaulay.

\smallskip
There are only three possibilities for $C'_0$.

\smallskip
$\bullet$ Type I.

\smallskip
For every surface $S \in \vert \mathcal{O}_{\mathbb{P}^3}(k) \vert$
with $1\leq k\leq \frac{4}{3}(n+1)$, one has either $\dim S\cap
C'_0=0$, or $\mathrm{dim}\ S\cap C'_0 = 1$ and
$$
\deg(S\cap C'_0) \leq k\sqrt{2d}-\dfrac{k^2}{2}.
$$
In view of  Corollary \ref{Halphen2} we deduce (set $p=n+1$):
$$
p_a(C'_0)\leq \frac{d^2}{2(n+1)}+\frac{2}{9}(2n-7)d+1.
$$
We are done, because this number is $\leq
\frac{d^2}{2n}+\frac{d}{2n}\left(2\pi_X-2-n\right)-\frac{1}{n}\sqrt{d}$
when $d>M^*(n)$.

\smallskip
$\bullet$ Type II.

\smallskip
$C'_0$ is contained in a surface  $S_0\subset \mathbb P^3$ with
$\deg S_0\in [1,\,\frac{4}{3}(n+1)]$. We may lift such a surface to
a surface $D_0\subset X$ containing $C_0$, with $\deg D_0\in
[n,\,\frac{4}{3}n(n+1)]$. Let
$s_0$ be the minimal degree of an integral component of $D_0$.
Observe that $n\leq s_0$ because $X$ is factorial and the Picard
group of $X$ is generated by the hyperplane section. From
Proposition \ref{nSnonintegra} we know that
$$
p_a(C_0)\leq \frac{d^2}{2s_0}+\frac{d}{2s_0}(2\pi_0-2-s_0)+4(\deg
D_0)^6.
$$
Here $\pi_0$ denotes the sectional genus of such a component of
minimal degree. If $s_0>n$, a direct computation proves that, for
$d>M^*(n)$, one has
$$
p_a(C_0)\leq
\frac{d^2}{2s_0}+\frac{d}{2s_0}(2\pi_0-2-s_0)+4(\deg
D_0)^6\leq
\frac{d^2}{2n}+\frac{d}{2n}\left(2\pi_X-2-n\right)-\frac{1}{n}\sqrt{d},
$$
and we are done. If $s_0=n$, our claim follows by Proposition
\ref{refined}.

\smallskip
$\bullet$ Type III.

\smallskip
$C'_0$ is not of Type I and is not of Type II.

Therefore, there exists a surface  $S_1\subset \mathbb P^3$ with
$\deg S_1=k_1\in [1,\,\frac{4}{3}(n+1)]$,  such that $S_1\cap C'_0$
is a curve and such that
$$
\deg S_1\cap C'_0>  k_1\sqrt{2d}-\dfrac{{k_1}^2}{2}.
$$
Observe that, since $C'_0$ is Cohen-Macaulay,  we have $\deg
C'_0>\deg S_1\cap C'_0$. In fact, otherwise,  $S_1\cap C'_0=C'_0$,
and $C'_0$ would be contained in $S_1$, i.e. $C'_0$ would be of Type
II. As in the previous proof of Theorem \ref{n3fold},
combining \cite[Lemma 3.1]{Liu} with \cite[Proposition
3.15]{Liu}, we deduce the existence of a  curve $C_1\subset C_0$, a
surface $D_1\subset X$  with $\deg D_1=k_1n\in
[n,\,\frac{4}{3}n(n+1)]$, such that
\begin{equation*}
p_a(C_0)=p_a(C_1)+p_a(D_1\cap C_0)+k_1d_1-1,
\end{equation*}
$$
0>\deg (D_1\cap C_0) >  k_1\sqrt{2d}-\dfrac{{k_1}^2}{2},
$$
and
$$
\deg C_1+\deg(D_1\cap C_0)=d=\deg C_0.
$$
Set
$$
d_1=\deg C_1,\quad d'_0=d_0-d_1=d-d_1=\deg (D_1\cap C_0).
$$
As in the proof of Theorem \ref{n3fold}, we have:
$$
p_a(C_0)=p_a(C_1)+p_a(C_0\cap D_1)+k_1d_1-1
$$
$$
\leq
\frac{{d'_0}^2}{2n}+\frac{d'_0}{2n}(2\pi_X-2-n)+4\nu^6+[p_a(C_1)+k_1d_1]
$$
$$
=\frac{({d-d_1})^2}{2n}+\frac{d-d_1}{2n}(2\pi_X-2-n)+4\nu^6+[p_a(C_1)+k_1d_1]
$$
$$
=\frac{{d}^2}{2n}+\frac{d}{2n}\left(2\pi_X-2-n\right)-\frac{1}{n}\sqrt{d}+R^*_1
$$
where
$$
R^*_1=4\nu^6+\frac{1}{n}\sqrt{d}-\frac{dd_1}{n}+\frac{d_1^2}{2n}-\frac{d_1}{2n}(2\pi_X-2-n)+p_a(C_1)+k_1d_1.
$$
It suffices to prove that $R^*_1\leq 0$. Let's distinguish two
cases:

\smallskip
1) $d_1>M(n)$;

\smallskip
2) $d_1\leq M(n)$.

\smallskip
In the case 1), by Theorem \ref{n3fold} we have
$$
p_a(C_1)\leq \frac{{d_1}^2}{2n}+\frac{d_1}{2n}(2\pi_X-2-n)+4\nu^6.
$$
Inserting into $R^*_1$ we have:
$$
R^*_1\leq
4\nu^6+\frac{1}{n}\sqrt{d}-\frac{dd_1}{n}+\frac{d_1^2}{2n}-\frac{d_1}{2n}(2\pi_X-2-n)+\left(\frac{{d_1}^2}{2n}+\frac{d_1}{2n}(2\pi_X-2-n)+4\nu^6\right)+k_1d_1
$$
$$
\leq
\frac{1}{n}\sqrt{d}-\frac{dd_1}{n}+\frac{d_1^2}{n}+8\nu^6+k_1d_1.
$$
Hence, it suffices to prove that
$$
\frac{d_1}{n}(d-d_1)\geq  \frac{1}{n}\sqrt{d}+8\nu^6+k_1d_1.
$$
Dividing by $\frac{d_1}{n}$, taking into account that $d_1>M(n)$ and
that $d-d_1>k_1\sqrt{2d}-\dfrac{{k_1}^2}{2}$, it suffices that:
$$
\sqrt{2d}\geq \sqrt{d}+2\nu.
$$
This holds true when $d>M^*(n)$.

\smallskip
In the case 2), i.e. $d_1\leq M(n)$, using Lemma \ref{planegenus} and (\ref{stimapi}),
we have:
$$
R^*_1=4\nu^6+\frac{1}{n}\sqrt{d}-\frac{dd_1}{n}+\frac{d_1^2}{2n}-\frac{d_1}{2n}(2\pi_X-2-n)+p_a(C_1)+k_1d_1
$$
$$
=4\nu^6+\frac{1}{n}\sqrt{d}-\frac{d'_0d_1}{n}-\frac{d_1^2}{2n}-\frac{d_1}{2n}(2\pi_X-2-n)+p_a(C_1)+k_1d_1
$$
$$
\leq
4\nu^6+\frac{1}{n}\sqrt{d}-\frac{d'_0d_1}{n}-\frac{d_1^2}{2n}+2d_1+\frac{1}{2}(d_1-1)(d_1-2)+k_1d_1.
$$
It suffices to prove that
$$
\frac{d_1}{n}(d-d_1)\geq
4\nu^6+\frac{1}{n}\sqrt{d}+2d_1+\frac{1}{2}(d_1-1)(d_1-2)+k_1d_1.
$$
Dividing by $\frac{d_1}{n}$, taking into account that
$d-d_1>k_1\sqrt{2d}-\dfrac{{k_1}^2}{2}$ and $d_1\leq M(n)$, it
suffices that
$$
\sqrt{2d}\geq
\sqrt{d}+4n\nu^6+2n+\frac{n}{2}M(n)+n+\frac{2}{3}(n+1).
$$
Hence, it suffices that
$$
\sqrt{2d}\geq \sqrt{d}+nM(n).
$$
This holds true for $d>M^*(n)$.
\end{proof}

\smallskip
\begin{proof}[Proof of Corollary \ref{cor2}] Let $C\in \mathcal C(X,d)$
be a curve of maximal arithmetic genus $G(X,d)$. By the proof of
Corollary \ref{cor1} and by our assumption on $p_a(C)$, it follows
that $C$ is contained in a surface $S$ of $X$ of degree $\deg S\leq
\frac{4}{3}n(n+1)$, with an integral component of minimal degree
$n$, i.e. a hyperplane section of $X$ (we may assume $X\subset
\mathbb P^r$ and  $h^0(X,\mathcal O_X(1))=r+1$). Therefore, our
claim follows by Corollary \ref{intcompl}.
\end{proof}

\smallskip
\begin{remark}\label{finalr}
$(i)$ Corollary \ref{ncor3} is an immediate consequence of Corollary
\ref{cor1}.

\smallskip
$(ii)$ With the same notations of Theorem \ref{n3fold}, recall that
if $C$ is a complete intersection of degree $d$ on a hyperplane
section of $X$, then
$$
p_a(C)=\frac{{d}^2}{2n}+\frac{d}{2n}(2\pi_X-2-n)+1.
$$
This proves that, apart the constant term, the bound
(\ref{nbound3fold}) is sharp (compare with (\ref{boundintcompl}), Remark \ref{Halphen}, $(i)$,  (\ref{Sab2})).
Moreover, in this case Corollary
\ref{cor2} applies. Therefore, at least when $d\gg 0$ is a multiple
of $n$, a curve of $X$ of maximal arithmetic genus is necessarily
contained in a hyperplane section of $X$.

\smallskip
$(iii)$ It may be interesting to point out that, if $C$ is an
{\it integral} (i.e. reduced and irreducible) curve of degree $d\gg n$,
one can easily prove Theorem \ref{Liu} and Theorem
\ref{n3fold} combining the main
result in \cite{CCD} with \cite[Lemma]{D3}. We briefly sketch the proof. If $C$ is not
contained in a surface of $\mathbb P^r$ of degree $<n+1$, by
\cite{CCD} we know that
$$
p_a(C)\leq \frac{{d}^2}{2(n+1)}+O(d).
$$
And this number is strictly less than (\ref{nbound3fold}) for $d\gg n$.
Therefore, one may assume that $C$ is contained in a surface of
$\mathbb P^r$ of degree $\leq n$. Since $C$ is contained in $X$, by
our hypotheses on $X$ and degree reasons, $C$ should be contained in
a hyperplane section $H$ of $X$. Since $H$ is integral, inequality
(\ref{nbound3fold}) follows from \cite[Lemma]{D3}. If $H$ is a Castelnuovo surface (e.g. $H\subset\mathbb P^3$),
instead of the inequality (\ref{nbound3fold})
one may apply (\ref{MainCCD}), which is sharp (when $H\subset\mathbb P^3$, (\ref{MainCCD}) reduces to
Halphen's bound  (\ref{HGPbound}) or (\ref{Liubound})). Moreover,
if $X\subset \mathbb P^5$ is a complete intersection, for the integral curves in $H$ one may apply
the sharp bound appearing in \cite[Theorem, p. 119]{P4}.

\smallskip
$(iv)$ In view of previous remark, one might ask whether the
previous argument can be adapted to non integral curves.
\end{remark}

\end{document}